\documentclass[EJP]{ejpecp}
\usepackage[T1]{fontenc}
\usepackage[utf8]{inputenc}

\SHORTTITLE{Joint CLT for random sesquilinear forms} 

\TITLE{Joint CLT for several random sesquilinear forms with~applications~to
  large-dimensional spiked~population~models\thanks{Supported by the National
    Natural Science Foundation of China [11071213, 11371317] and Research
    Grants Council, Hong Kong SAR, China [781511M, 705413P].}} %

\AUTHORS{Wang~Qinwen\footnote{Department of Mathematics, Zhejiang University,
    China. \EMAIL{wqw8813@gmail.com}}\and Su~Zhonggen \footnote{Department of
    Mathematics, Zhejiang University, China. \EMAIL{suzhonggen@zju.edu.cn}}
  \and 
  Yao~Jianfeng\footnote{Department of Statistics and Actuarial Science, The
    University of Hong Kong, Pokfulam, Hong Kong. \BEMAIL{jeffyao@hku.hk}}}


\AMSSUBJ{60F05} 
\AMSSUBJSECONDARY{60B20} 

\SUBMITTED{February 23, 2014} 
\ACCEPTED{October 27, 2014} 

\ARXIVID{1402.6064} 

\VOLUME{19}
\YEAR{2014}
\PAPERNUM{103}
\DOI{v19-3339}

\ABSTRACT{In this paper, we derive a joint central limit theorem for random
  vector whose components are function of random sesquilinear forms. This
  result is a natural extension of the existing central limit theory on random
  quadratic forms. We also provide applications in random matrix theory
  related to large-dimensional spiked population models. For the first
  application, we find the joint distribution of grouped extreme sample
  eigenvalues correspond to the spikes. And for the second application, under
  the assumption that the population covariance matrix is diagonal with $k$
  (fixed) simple spikes, we derive the asymptotic joint distribution of the
  extreme sample eigenvalue and its corresponding sample eigenvector
  projection.}

\newcommand\limm{\underset{n \rightarrow \infty}{\lim}}
\newcommand\limn{\underset{n \rightarrow \infty}{\lim}}
\newcommand\E{\mathbb{E}}
\newcommand\cov{\mathop{\text{Cov}}}
\newcommand\tr{\mathop{\text{tr}}}
\newcommand\diag{\mathop{\text{diag}}}
\newcommand\N{\mathcal{N}}

\begin{document}

\section{Introduction}
The aim of this paper is to derive the joint central limit theorem of a new type of random vector whose components are made with several groups of random sesquilinear forms.
To be more specific, we consider a sequence $\big\{(x_i, y_i)_{i \in \mathbb{N}}\big\}$ of iid. complex-valued, zero-mean random vector belonging to $\mathbb{C}^{K}\times \mathbb{C}^{K}$ ($K$ fixed) with a finite moment of  fourth-order. For positive integer $n\geq1$, write \begin{eqnarray}\label{1}
x_i=(x_{1i}, \cdots, x_{Ki})^T,~~~~ X(l)=(x_{l1}, \cdots, x_{ln})^T~ (1\leq l\leq K)~,
\end{eqnarray}
with a similar definition for the vectors $\{y_i\}$ and $\{Y(l)\}_{1\leq l\leq K}$. The covariance between  $x_{l1}$ and $y_{l1}$ is denoted as  $\rho(l)=E[\overline{x}_{l1}y_{l1}]$, $1\leq l\leq K$. Let $\big\{A_n=[a_{ij}(n)]\big\}_n$ and $\big\{B_n=[b_{ij}(n)]\big\}_n$ be two sequences of  $n \times n$ Hermitian matrices, and define
\begin{eqnarray}
&&U(l):=\frac{1}{\sqrt{n}}\big[X(l)^{*}A_n Y(l)-\rho(l)tr A_n\big]~, \label{u}\\
&&V(l):=\frac{1}{\sqrt{n}}\big[X(l)^{*}B_n Y(l)-\rho(l)tr B_n\big]~.\nonumber
\end{eqnarray}
We are studying  the  joint central limit theorem of the $2K$-dimensional complex-valued random vector:
\[
\big(U(1), \cdots, U(K), V(1), \cdots, V(K)\big)^{T}~.
\]

If we use only one sequence of Hermitian matrix, say $\{A_n\}$ and consider one form ($K=1$), then the problem reduces to the central limit theorem of a simple random  sesquilinear form:
\[
U(1):=\frac{1}{\sqrt{n}}\big[X(1)^{*}A_n Y(1)-\rho(1)tr A_n\big]~.
\]
If we further impose $Y\equiv X$, we obtain a classical random quadratic form
\[
U^{*}(1):=\frac{1}{\sqrt{n}}\big[X(1)^{*}A_n X(1)-\rho(1)tr A_n\big]
\]
with independent random variables.

There exists an extensive literature on the asymptotic distribution of quadratic form $U^{*}(1)$. The pioneering work in this area dates back to \cite{Sevastyanov}, who deals principally with the case when the variables $X$ have normal distribution.
This CLT is extended to arbitrary iid. components in $X$ by \cite{Whittle}, with additional conditions on the matrix $A$: in particular, $A$ has a zero diagonal (i.e quadratic form: $\tilde{U}(1):=\frac{1}{\sqrt{n}}X(1)^{*}A_n X(1)$). Later extensions deal with other types of limiting theorem (functional CLT, law of iterated logarithm) or dependent random variables in $X$, see: \cite{Rotar}, \cite{De}, \cite{fox}, \cite{mi} and \cite{Ja94} for reference.

In a different area, \cite{guang} and \cite{Hachem} established the asymptotic behavior of quadratic form and bilinear form, where $A=S_n$ is a sample covariance matrix and $A=(M_n-zI)^{-1}$ is the resolvent of some large dimensional random matrix $M_n$, respectively. Such CLT can be used in the areas of wireless communications and electrical engineering.

In the paper of \cite{BaiYao08}, the authors derived the central limit theorem for  $U(l)$  in \eqref{u} (i.e with one group of sesquilinear forms) in their Appendix as a tool for establishing the central limit theory for the extreme  sample eigenvalues when the population has a spiked covariance structure. 

In this paper, we follow the lines and strategy that was put forward in \cite{BaiYao08}, and extend this CLT to arbitrary number of groups of random sesquilinear forms, which is presented in Section \ref{main}. Indeed, this extension has been motivated by  applications  in the field of random matrix theory related to the spiked population model. When the population has a spiked covariance structure, we establish the asymptotic joint distribution of any two groups of extreme sample eigenvalues that correspond to the spikes. Besides, when the population covariance matrix is diagonal with $k$ (fixed) simple spikes, we find the joint distribution of the extreme sample eigenvalue and its corresponding sample eigenvector projection using our main result. All these applications are developed in Section \ref{application}.  Section \ref{proof} and the last Section contain  proofs and some additional technical  lemmas.
\section{Main result: central limit theorem for random sesquilinear forms}\label{main}

\begin{theorem}\label{sesquilinear}
Let $\big\{A_n=[a_{ij}(n)]\big\}_n$ and $\big\{B_n=[b_{ij}(n)]\big\}_n$ be two sequences of  $n \times n$ Hermitian matrices and the vector $\{X(l), Y(l)\}_{1 \leq l \leq K}$ be defined as in \eqref{1}. Assume that the following limits exist:
\begin{eqnarray*}
&&w_1=\limm\frac{1}{n}\tr [A_n\circ A_n],
  w_2=\limm \frac 1n\tr [B_n\circ B_n],
 w_3=\limm\frac{1}{n}\tr [A_n\circ B_n],~\label{w3}\\
&&\theta_1=\limm\frac{1}{n}\tr [A_n A_n^{*}],~~~
  \theta_2=\limm\frac{1}{n}\tr [B_n B_n^{*}],~~~
 \theta_3=\limm\frac{1}{n}\tr [A_n B_n^{*}],~\label{th3}\\
&&\tau_1=\limm\frac{1}{n}\tr [A_n^2],\quad\quad~
  \tau_2=\limm\frac{1}{n}\tr [B_n^2],\quad\quad~
 \tau_3=\limm\frac{1}{n}\tr [A_n B_n],~\label{t3}
\end{eqnarray*}
where $A\circ B$ denotes the Hadamard product of two matrices $A$ and $B$, i.e. $(A \circ B)_{ij}=A_{ij}\cdot B_{ij}$.
Define two groups of sesquilinear forms:
\begin{eqnarray*}
U(l)=\frac{1}{\sqrt{n}}\big[X(l)^{*}A_n Y(l)-\rho(l)tr A_n\big]~,
V(l)=\frac{1}{\sqrt{n}}\big[X(l)^{*}B_n Y(l)-\rho(l)tr B_n\big]~.
\end{eqnarray*}
Then, the $2K$-dimensional complex-valued random vector:
\[
\big(U(1), \cdots, U(K), V(1), \cdots, V(K)\big)^{T}
\]
converges weakly to a zero-mean complex-valued vector $W$ whose real and imaginary parts are Gaussian. Moreover, the Laplace transform of $W$ is given by
\[
\E\exp~\Bigg(\begin{pmatrix}
      c \\
      d \\
    \end{pmatrix}^TW
\Bigg)=\exp~\Bigg[\frac12\begin{pmatrix}
               c \\
               d \\
             \end{pmatrix}^TB\begin{pmatrix}
                               c & d\\
                             \end{pmatrix}
\Bigg]~,~~~~c,d \in \mathbb{C}^K~,
\]
with
\begin{equation*}
B=\left(
\begin{array}{cc}
 B_{11} & B_{12}\\
 B_{12} & B_{22}\\
 \end{array}
\right)_{2K \times 2K}.
\end{equation*}
Each block within $B$ is a $K \times K$ matrix, having the structure ($l, l^{'}=1, \cdots, K$):
\begin{eqnarray*}
&&B_{11}(l,l^{'})=\cov ~(U(l), U(l^{'}))=w_1A_1+(\tau_1-w_1)A_2+(\theta_1-w_1)A_3~,\\
&&B_{22}(l,l^{'})=\cov~ (V(l), V(l^{'}))=w_2A_1+(\tau_2-w_2)A_2+(\theta_2-w_2)A_3~,\\
&&B_{12}(l,l^{'})=\cov~ (U(l), V(l^{'}))=w_3A_1+(\tau_3-w_3)A_2+(\theta_3-w_3)A_3~,
\end{eqnarray*}
where $A_1$, $A_2$ and $A_3$ are given by
\begin{eqnarray}
&&A_1=\E(\overline{x}_{l1}y_{l1}\overline{x}_{l^{'}1}y_{l^{'}1})-\rho(l)\rho(l^{'})~,\\
&&A_2=\E(\overline{x}_{l1}\overline{x}_{l^{'}1})E(y_{l1}y_{l^{'}1})~,\\
&&A_3=\E(\overline{x}_{l1}y_{l^{'}1})E(\overline{x}_{l^{'}1}y_{l1})~.\label{aaa}
\end{eqnarray}
\end{theorem}

\begin{proof}(proof of Theorem \ref{sesquilinear})
It is sufficient to establish the CLT for the linear combinations of  random Hermitian sesquilinear forms:
\[
\sum_{l=1}^{K}[c_lX(l)^{*}A_nY(l)+d_lX(l)^{*}B_nY(l)]~,
\]
where the coefficients $(c_l), (d_l) \in \mathbb{C}^K \times \mathbb{C}^K$ are arbitrary. Also, it holds that
\begin{eqnarray*}
\E[X(l)^{*}A_nY(l)]=\rho(l)tr A_n~,~~~~\E[X(l)^{*}B_nY(l)]=\rho(l)tr B_n~.
\end{eqnarray*}
We use the moment method as in \cite{BaiYao08}. Consider the linear combination of the two sesquilinear forms
\begin{eqnarray*}
\eta_n=\frac{1}{\sqrt{n}}\sum_{l=1}^{K}\big\{c_l[X(l)^{*}A_n Y(l)-\rho(l)tr A_n]+d_l[X(l)^{*}B_n Y(l)-\rho(l)tr B_n]\big\}~,
\end{eqnarray*}
which can be expanded  as follows:
\begin{eqnarray*}
\eta_n&=&\frac{1}{\sqrt{n}}\sum_{l=1}^{K}\Big\{c_l \big[\sum_{u=1}^{n}(X(l)_{u}^{*}Y(l)_{u}-\rho(l))a_{uu}+\sum_{u \neq v}X(l)_{u}^{*}Y(l)_{v}a_{uv}\big]\\
&&+d_l \big[\sum_{u=1}^{n}(X(l)_{u}^{*}Y(l)_{u}-\rho(l))b_{uu}+\sum_{u \neq v}X(l)_{u}^{*}Y(l)_{v}b_{uv}\big]\Big\}\\
&=&\frac{1}{\sqrt{n}}\sum_{e=(u, v)}\Big\{\sum_{l=1}^{K}\big[(c_l\overline{x}_{lu}y_{lu}-c_l\rho(l))a_{uu}+c_l\overline{x}_{lu}y_{lv}a_{uv}\big]\\
&&+\sum_{l=1}^{K}\big[(d_l\overline{x}_{lu}y_{lu}-d_l\rho(l))b_{uu}+d_l\overline{x}_{lu}y_{lv}b_{uv}\big]\Big\}\\
&=&\frac{1}{\sqrt{n}}\sum_{e}(a_e \psi_e+b_e\varphi_e)~,
\end{eqnarray*}
where  $e$ is an edge associated with vertex $u$ and $v$, i.e. $e=(u, v)\in \{1, \cdots, n\}^2$; and
\begin{equation}\label{psi}
\psi_e\triangleq
 \left\{
  \begin{array}{cc}
   \sum_{l=1}^{K}c_l( \overline{x}_{lu} y_{lu}-\rho(l))~, & u=v~,\\
   \sum_{l=1}^{K}c_l \overline{x}_{lu} y_{lv}~, & u\neq v~,\\
  \end{array}
 \right .
\end{equation}
\begin{equation}\label{varphi}
\varphi_e\triangleq
 \left\{
  \begin{array}{cc}
   \sum_{l=1}^{K}d_l( \overline{x}_{lu} y_{lu}-\rho(l))~, & u=v~,\\
   \sum_{l=1}^{K}d_l \overline{x}_{lu} y_{lv}~, & u\neq v~.\\
  \end{array}
 \right .
\end{equation}
Then
\begin{eqnarray}\label{2}
n^{\frac{K}{2}}\eta_n^{K}&=&\sum_{e_1\cdots e_K}(a_{e_1}\psi_{e_1}+b_{e_1}\varphi_{e_1})\cdots (a_{e_K}\psi_{e_K}+b_{e_K}\varphi_{e_K})\\
&=&\sum_{G_1\bigcup G_2}a_{G_1}\psi_{G_1}b_{G_2}\varphi_{G_2}~,\nonumber
\end{eqnarray}
where
\begin{eqnarray*}
a_{G_1}=\prod_{e \in G_1}a_e~,~~\psi_{G_1}=\prod_{e \in G_1}\psi_e~,~~b_{G_2}=\prod_{e \in G_2}b_e~,~~
\varphi_{G_2}=\prod_{e \in G_2}\varphi_e~.
\end{eqnarray*}
To each sum in equation \eqref{2}, we associate a directed graph $G$ by drawing an arrow $u\rightarrow v$ for each factor $e_j=(u, v)$. We denote $G_1$ as a subgraph of $G$ corresponding to the coefficients being $a\psi$, and $G_2$ the remaining: $G_2=G \verb|\|
  G_1$. Besides, to a loop $u\rightarrow u$ corresponds the product $a_{uu}\psi_{uu}=a_{uu}\sum_{l=1}^{K}c_l( \overline{x}_{lu} y_{lu}-\rho(l))$ and to an edge $u\rightarrow v$ $(u \neq v)$ corresponds the product $a_{uv}\psi_{uv}=a_{uv}\sum_{l=1}^{K}c_l \overline{x}_{lu} y_{lv}$. The same holds for $b_{uu}\varphi_{uu}$ and $b_{uv}\varphi_{uv}$.

In the paper of \cite{BaiYao08} (proof of Theorem 7.1), they show that only three types of components in the graph $G$ contribute to a non-negligible term (see Figure \ref{3com}):
\begin{figure}[ht]
\begin{center}
\includegraphics[width=12cm]{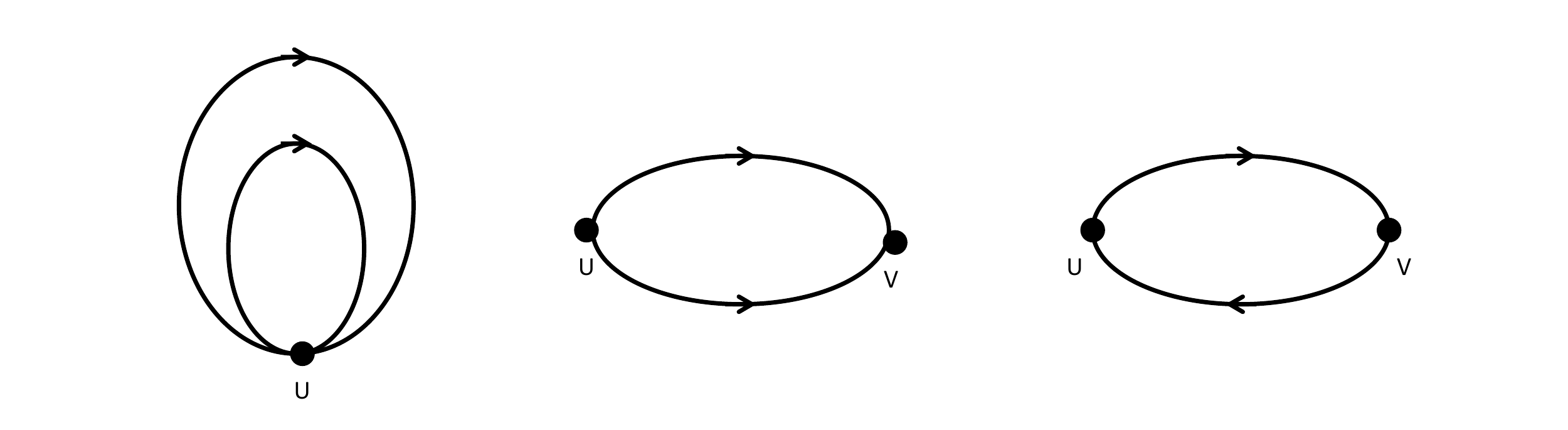}
\caption{three major components in the graph $G$}
\label{3com}
\end{center}
\end{figure}

Because $G_1$ and $G_2$ are subgraphs of $G$, and by the definition in equation \eqref{psi} and \eqref{varphi}, $\psi_e$ differs from $\varphi_e$ only through the coefficient $c_l$ or $d_l$ in front. So the difference between $\psi_e$ and $\varphi_e$ is at most $O(1)$, which means that for the components in the graph $G$ that have $o(1)$ contribution to $En^{K/2}\xi_n^{K}$ (see \cite{BaiYao08} for detail of $\xi_n$) should still have $o(1)$ contribution to $En^{K/2}\eta_n^{K}$. Based on this fact, we get this time that only the influence of the following nine components (in Figure \ref{9com})  counts.
The numbers $k_1, \cdots, k_9$ in Figure \ref{9com} stand for the multiplicity of each component, so by degree of each vertex, we also have the restriction that $4(k_1+\cdots +k_9)=2K$, which means $K$ should be an even number, denoted as $2p$ for convenience.

From the combinatorics, we have this time
\begin{align}\label{3}
\E n^{K/2}\eta_n^{K}&=\E\sum_{G_1\bigcup G_2}a_{G_1}\psi_{G_1}b_{G_2}\varphi_{G_2}\nonumber\\[1mm]
&=\sum_{2(k_1+\cdots +k_9)=K}\frac{\begin{small}\begin{pmatrix}
                                     K \\
                                     2 \\
                                   \end{pmatrix}\end{small}
\begin{small}\begin{pmatrix}
                                     K-2 \\
                                     2 \\
                                   \end{pmatrix}\end{small}\cdots \begin{small}\begin{pmatrix}
                                     2 \\
                                     2 \\
                                   \end{pmatrix}\end{small} \cdot 2^{k_3+k_6+k_9}}{k_1!\cdots k_9!}\times D_1D_2\cdots D_9+o(n^{K/2})\nonumber\\[1mm]
&=\sum_{k_1+\cdots +k_9=p}\frac{(2p)!\cdot 2^{k_3+k_6+k_9}}{2^p\cdot k_1!\cdots k_9!}\times D_1D_2\cdots D_9+o(n^{K/2})~.
\end{align}
\begin{figure}[ht]
\centering
\includegraphics[width=12cm]{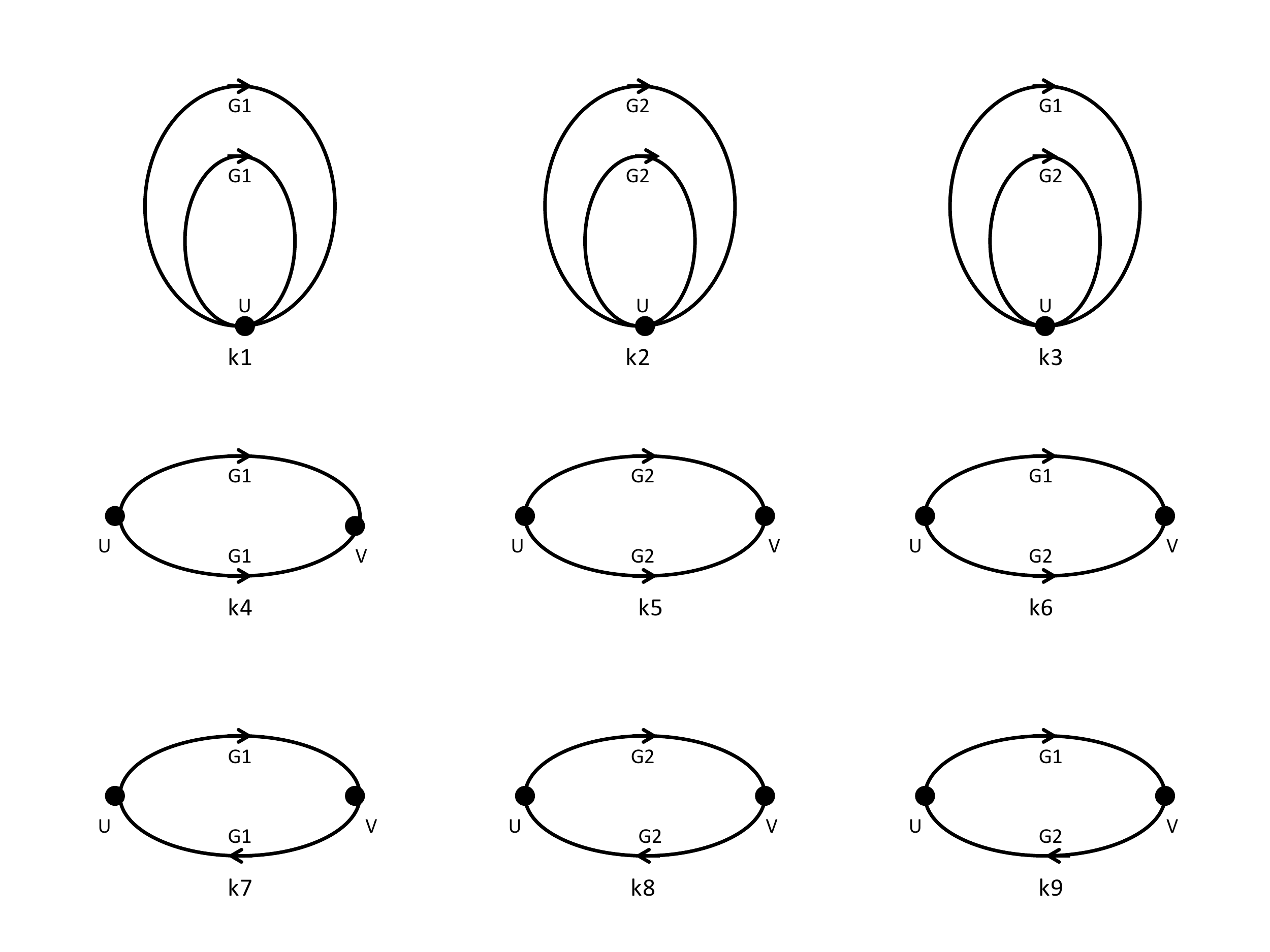}
\caption{nine major components in the graph $G_1\bigcup G_2$}\label{9com}
\end{figure}
The coefficients in front of $D_1D_2\cdots D_9$ is due to the fact that by observing the nine components in Figure \ref{9com}, we find that each component is made of two edges; first we combine two edges in a group in the total of $K$ edges, that is $\begin{tiny}\begin{pmatrix}
                                     K \\
                                     2 \\
                                   \end{pmatrix}\end{tiny}
\begin{tiny}\begin{pmatrix}
                                     K-2 \\
                                     2 \\
                                   \end{pmatrix}\end{tiny}\cdots \begin{tiny}\begin{pmatrix}
                                     2 \\
                                     2 \\
                                   \end{pmatrix}\end{tiny}$; second, the first $k_1$ (also the following $k_2, \cdots, k_9$) groups should be the same, we must exclude the $k_1!\cdots k_9!$ possibilities from the total of $\begin{tiny}\begin{pmatrix}
                                     K \\
                                     2 \\
                                   \end{pmatrix}\end{tiny}
\begin{tiny}\begin{pmatrix}
                                     K-2 \\
                                     2 \\
                                   \end{pmatrix}\end{tiny}\cdots \begin{tiny}\begin{pmatrix}
                                     2 \\
                                     2 \\
                                   \end{pmatrix}\end{tiny}$; and last, for the three components in the last column of Figure \ref{9com}, the two edges in each component belong to different subgraphs (one edge in $G_1$ and the other in $G_2$), so there should be an additional perturbation $2^{k_3+k_6+k_9}$ added, and combine all these facts leads to the result.

Then we specify the terms of $D_1, D_2, \cdots, D_9$ in the following:
\begin{eqnarray*}
D_1&=&\prod_{j=1}^{k_1}\E\Big[a_{u_ju_j}^2\big\{\sum_{l=1}^Kc_l\big(\overline{x}_{lu_j}y_{lu_j}-\rho(l)\big)\big\}^2\Big]\\
&=&\prod_{j=1}^{k_1}a_{u_ju_j}^2\sum_{l, l^{'}}c_l c_{l^{'}}\big[E(\overline{x}_{l1}y_{l1}\overline{x}_{l^{'}1}y_{l^{'}1})-\rho(l)\rho(l^{'})\big]\\
&=&\prod_{j=1}^{k_1}a_{u_ju_j}^2\sum_{l, l^{'}}c_l c_{l^{'}}A_1\\
&\triangleq&\prod_{j=1}^{k_1}a_{u_ju_j}^2 \alpha_1~.
\end{eqnarray*}
Similarly, we have:
\begin{eqnarray*}
D_2&=&\prod_{j=1}^{k_2}\E\Big[b_{u_ju_j}^2\big\{\sum_{l=1}^Kd_l\big(\overline{x}_{lu_j}y_{lu_j}-\rho(l)\big)\big\}^2\Big]\\
&\triangleq&\prod_{j=1}^{k_2}b_{u_ju_j}^2 \beta_1~,
\end{eqnarray*}
\begin{eqnarray*}
D_3&=&\prod_{j=1}^{k_3}\E\Big[a_{u_ju_j}b_{u_ju_j}\sum_{l=1}^Kc_l\big(\overline{x}_{lu_j}y_{lu_j}-\rho(l)\big)\sum_{l=1}^Kd_l\big(\overline{x}_{lu_j}y_{lu_j}-\rho(l)\big)\Big]\\
&\triangleq&\prod_{j=1}^{k_3}a_{u_ju_j}b_{u_ju_j} \gamma_1~,
\end{eqnarray*}
\begin{eqnarray*}
D_4&=&\prod_{j=1}^{k_4}\E\Big[a_{u_jv_j}^2\big(\sum_{l=1}^Kc_l\overline{x}_{lu_j}y_{lv_j}\big)^2\Big]\\
&=&\prod_{j=1}^{k_4}a_{u_jv_j}^2\sum_{l, l^{'}}c_l c_{l^{'}}E\big(\overline{x}_{l1}\overline{x}_{l^{'}1}\big)E\big(y_{l1}y_{l^{'}1}\big)\\
&=&\prod_{j=1}^{k_4}a_{u_jv_j}^2\sum_{l, l^{'}}c_l c_{l^{'}}A_2\\
&\triangleq&\prod_{j=1}^{k_4}a_{u_jv_j}^2 \alpha_2~,
\end{eqnarray*}
\begin{eqnarray*}
D_5&=&\prod_{j=1}^{k_5}\E\Big[b_{u_jv_j}^2\big(\sum_{l=1}^Kd_l\overline{x}_{lu_j}y_{lv_j}\big)^2\Big]\\
&\triangleq&\prod_{j=1}^{k_5}b_{u_jv_j}^2 \beta_2~,
\end{eqnarray*}
\begin{eqnarray*}
D_6&=&\prod_{j=1}^{k_6}\E\Big[a_{u_jv_j}b_{u_jv_j}\big(\sum_{l=1}^Kc_l\overline{x}_{lu_j}y_{lv_j}\big)\big(\sum_{l=1}^Kd_l\overline{x}_{lu_j}y_{lv_j}\big)\Big]\\
&\triangleq&\prod_{j=1}^{k_6}a_{u_jv_j}b_{u_jv_j} \gamma_2~,
\end{eqnarray*}
\begin{eqnarray*}
D_7&=&\prod_{j=1}^{k_7}\E\Big[|a_{u_jv_j}|^2\big(\sum_{l=1}^Kc_l\overline{x}_{lu_j}y_{lv_j}\big)\big(\sum_{l=1}^Kc_l\overline{x}_{lv_j}y_{lu_j}\big)\Big]\\
&=&\prod_{j=1}^{k_7}|a_{u_jv_j}|^2\sum_{l, l^{'}}c_l c_{l^{'}}E\big(\overline{x}_{l1}y_{l^{'}1}\big)E\big(\overline{x}_{l^{'}1}y_{l1}\big)\\
&=&\prod_{j=1}^{k_7}|a_{u_jv_j}|^2\sum_{l, l^{'}}c_l c_{l^{'}}A_3\\
&\triangleq&\prod_{j=1}^{k_7}|a_{u_jv_j}|^2\alpha_3~,
\end{eqnarray*}
\begin{eqnarray*}
D_8&=&\prod_{j=1}^{k_8}\E\Big[|b_{u_jv_j}|^2\big(\sum_{l=1}^Kd_l\overline{x}_{lu_j}y_{lv_j}\big)\big(\sum_{l=1}^Kd_l\overline{x}_{lv_j}y_{lu_j}\big)\Big]\\
&\triangleq&\prod_{j=1}^{k_8}|b_{u_jv_j}|^2\beta_3~,
\end{eqnarray*}
\begin{eqnarray*}
D_9&=&\prod_{j=1}^{k_9}\E\Big[a_{u_jv_j}b_{v_ju_j}\big(\sum_{l=1}^Kc_l\overline{x}_{lu_j}y_{lv_j}\big)\big(\sum_{l=1}^Kd_l\overline{x}_{lv_j}y_{lu_j}\big)\Big]\\
&\triangleq&\prod_{j=1}^{k_9}a_{u_jv_j}b_{v_ju_j}\gamma_3~.
\end{eqnarray*}
Combine these nine terms with equation \eqref{3}, we have
\begin{eqnarray*}
\E\eta_n^{2p}&=&n^{-p}\sum_{k_1+\cdots +k_9=p}\frac{(2p)!\cdot 2^{k_3+k_6+k_9}}{2^p\cdot k_1!\cdots k_9!}\prod_{(j_1\cdots j_9)=(1\cdots 1)}^{(k_1\cdots k_9)} a_{u_{j_1}u_{j_1}}^2\alpha_1^{k_1}b_{u_{j_2}u_{j_2}}^2\\
&&\times \beta_1^{k_2}
a_{u_{j_3}u_{j_3}}b_{u_{j_3}u_{j_3}}\gamma_1^{k_3} a_{u_{j_4}v_{j_4}}^2\alpha_2^{k_4}b_{u_{j_5}v_{j_5}}^2\beta_2^{k_5} a_{u_{j_6}v_{j_6}}b_{u_{j_6}v_{j_6}}\gamma_2^{k_6}\\
&&\times|a_{u_{j_7}v_{j_7}}|^2\alpha_3^{k_7}|b_{u_{j_8}v_{j_8}}|^2\beta_3^{k_8}a_{u_{j_9}v_{j_9}}b_{v_{j_9}u_{j_9}}\gamma_3^{k_9}+o(1)\\
&=&\frac{(2p-1)!!}{n^p}\big(\alpha_1\sum_{u=1}^n a_{uu}^2+\beta_1\sum_{u=1}^n b_{uu}^2+2\gamma_1\sum_{u=1}^n a_{uu}b_{uu}+\alpha_2\sum_{u\neq v} a_{uv}^2\\
&&+\beta_2\sum_{u\neq v}b_{uv}^2+2\gamma_2\sum_{u\neq v} a_{uv}b_{uv}+
\alpha_3\sum_{u\neq v} |a_{uv}|^2+\beta_3\sum_{u\neq v}|b_{uv}|^2\\
&&+2\gamma_3\sum_{u\neq v} a_{uv}b_{vu}
\big)^p+o(1)~,
\end{eqnarray*}
which means that $\eta_n\Longrightarrow \mathcal{N}(0, \sigma^2)$ by the moment method, with
\begin{eqnarray*}\label{4}
\sigma^2&=&\lim_{n\rightarrow \infty}\frac{1}{n}\Big[\alpha_1\sum_{u=1}^n a_{uu}^2+\beta_1\sum_{u=1}^n b_{uu}^2+2\gamma_1\sum_{u=1}^n a_{uu}b_{uu}+\alpha_2\sum_{u\neq v} a_{uv}^2+\beta_2\sum_{u\neq v}b_{uv}^2\nonumber\\
&&+2\gamma_2\sum_{u\neq v} a_{uv}b_{uv}+
\alpha_3\sum_{u\neq v} |a_{uv}|^2+\beta_3\sum_{u\neq v}|b_{uv}|^2+2\gamma_3\sum_{u\neq v} a_{uv}b_{vu}\Big]\nonumber\\
&=&\alpha_1w_1+\beta_1w_2+2\gamma_1w_3+\alpha_2(\tau_1-w_1)+\beta_2(\tau_2-w_2)+2\gamma_2(\tau_3-w_3)\nonumber\\
&&+\alpha_3(\theta_1-w_1)+\beta_3(\theta_2-w_2)+2\gamma_3(\theta_3-w_3)~\\
&=&\sum_{l,l^{'}}c_l c_{l^{'}}A_1w_1+\sum_{l,l^{'}}d_l d_{l^{'}}A_1w_2+2\sum_{l,l^{'}}c_l d_{l^{'}}A_1w_3+\sum_{l,l^{'}}c_l c_{l^{'}}A_2(\tau_1-w_1)\\
&&+\sum_{l,l^{'}}d_l d_{l^{'}}A_2(\tau_2-w_2)+2\sum_{l,l^{'}}c_l d_{l^{'}}A_2(\tau_3-w_3)+\sum_{l,l^{'}}c_l c_{l^{'}}A_3(\theta_1-w_1)\\
&&+\sum_{l,l^{'}}d_l d_{l^{'}}A_3(\theta_2-w_2)+2\sum_{l,l^{'}}c_l d_{l^{'}}A_3(\theta_3-w_3)\\
&=&\sum_{l,l^{'}}c_l c_{l^{'}}\big(A_1w_1+A_2(\tau_1-w_1)+A_3(\theta_1-w_1)\big)\\
&&+\sum_{l,l^{'}}d_l d_{l^{'}}\big(A_1w_2+A_2(\tau_2-w_2)+A_3(\theta_2-w_2)\big)\\
&&+2\sum_{l,l^{'}}c_l d_{l^{'}}\big(A_1w_3+A_2(\tau_3-w_3)+A_3(\theta_3-w_3)\big)~.
\end{eqnarray*}

The proof of  Theorem \ref{sesquilinear} is complete.
\end{proof}

\begin{corollary}\label{real}
Under the same conditions as in Theorem \ref{sesquilinear}, but with real random vectors $\{(x_i,y_i)_{i \in N}\}$, symmetric matrices $\{A_n=[a_{ij}(n)]\}_n$ and $\{B_n=[b_{ij}(n)]\}_n$,  the $2K$-dimensional real-valued random vector:
\[
\big(U(1), \cdots, U(K), V(1), \cdots, V(K)\big)^{T}
\]
converges weakly to a zero-mean $2K$-dimensional Gaussian vector with covariance matrix $B$.
\end{corollary}

Theorem \ref{sesquilinear} can be generalized to the joint distribution of several sesquilinear forms. We present this generalization in the following theorem. Recall that in the proof of Theorem \ref{sesquilinear}, we use the moment method and find the nine major components  presented in Figure \ref{9com},  which all contain two edges. Therefore, if now we consider the $k$ sesquilinear forms as a whole, there should be $\frac{3}{2}k(1+k)$ major components that will lead to a nonnegligible contribution. And each component still has two edges, from the same subgraph (both from $G_i$ $(i=1,\cdots,k)$ or from two different subgraphs (one from $G_i$ and the other from $G_j$ ($i\neq j$)). This means that the $k$ sesquilinear forms packed together only has pairwise covariance function. The proof for other steps is similar and omitted.
\begin{theorem}\label{ksesquilinear}
Let $\big\{A_m=[a_{ij}^{(m)}(n)]\big\}_n$ $m=(1,\cdots, k)$ be $k$ sequences of  $n \times n$ Hermitian matrices and the vector $\{X(l), Y(l)\}_{1 \leq l \leq K}$ are defined as \eqref{1}. Assume that the following limits exists ($m,m^{'}=(1, \cdots, k)$ and $m\neq m^{'}$):
\begin{eqnarray*}
&&w_m=\limm\frac{1}{n}\tr [A_m\circ A_m]~,\quad\quad
  w_{mm^{'}}=\limm \frac 1n\tr [A_m\circ A_{m^{'}}]~,\\
&&\theta_m=\limm\frac{1}{n}\tr [A_m A_m^{*}]~,\quad\quad\quad
  \theta_{mm^{'}}=\limm\frac{1}{n}\tr [A_m A_{m^{'}}^{*}]~,\\
&&\tau_m=\limm\frac{1}{n}\tr [A_m^2]~,\quad\quad\quad\quad~
  \tau_{mm^{'}}=\limm\frac{1}{n}\tr [A_m A_{m^{'}}]~.
\end{eqnarray*}
Denote the sesquilinear forms:
\begin{eqnarray*}
&&U^{(m)}(l)=\frac{1}{\sqrt{n}}\big[X(l)^{*}A_m Y(l)-\rho(l)tr A_m\big]~, m=1, \cdots, k,
\end{eqnarray*}
then the $(K\cdot k)$-dimensional complex-valued random vector:
\[
\big(U^{(1)}(1), \cdots, U^{(1)}(K), U^{(2)}(1), \cdots, U^{(2)}(K), U^{(k)}(1), \cdots, U^{(k)}(K)\big)^{T}
\]
converges weakly to a zero-mean complex-valued vector $W$ whose real and imaginary parts are Gaussian. Moreover, the Laplace transform of $W$ is given by
\[
\E\exp~\left(\begin{pmatrix}
      c_1 \\
      \vdots\\
      c_k \\
    \end{pmatrix}^TW
\right)=\exp~\left[\frac12\begin{pmatrix}
                  c_1 \\
                  \vdots\\
                  c_k \\
             \end{pmatrix}^TB\begin{pmatrix}
                               c_1 & \cdots & c_k\\
                             \end{pmatrix}
\right]~,~~~~c_i \in \mathbb{C}^K~,
\]
where $B$ could be written as \begin{equation*}
B=\left(
\begin{array}{cccc}
 B_{11} & B_{12} & \cdots & B_{1k}\\
 B_{21} & B_{22} & \cdots & B_{2k}\\
 \vdots & \vdots & \ddots & \vdots\\
 B_{k1} & B_{k2} & \cdots & B_{kk}\\
 \end{array}
\right)_{(K\cdot k)\times (K\cdot k)},
\end{equation*}
each block is a $K \times K$ matrices  with entries  (for $l, l^{'}=1, \cdots, K$):
\begin{eqnarray}
&&B_{ii}(l,l^{'})=\cov ~(U^i(l), U^i(l^{'}))=w_i A_1+(\tau_i-w_i)A_2+(\theta_i-w_i)A_3~,\label{bii}\\
&&B_{ij}(l,l^{'})=\cov~ (U^i(l), U^j(l^{'}))=w_{ij}A_1+(\tau_{ij}-w_{ij})A_2+(\theta_{ij}-w_{ij})A_3~, \nonumber \\\label{bij}
\end{eqnarray}
and $A_1$ $A_2$ and $A_3$ are the same as \eqref{aaa}.
\end{theorem}

Here we give an application  related to the existing literature on large-dimensional covariance matrices. In \cite{guang}, they establish the central limit theorem of the random quadratic forms $s_1^T(SS^T)^is_1$, where $S=(s_1,\cdots, s_k)$, $s_i=\frac{1}{\sqrt n}(v_{i1} \cdots v_{in})^T$, $\{v_{ij}\}$ are i.i.d. with  $\E v_{11}=0~, \E v^2_{11}=1$, $\E v^4_{11}=\nu_4 < \infty$. This $s_1^T(SS^T)^is_1$ can be written as a linear combination of a series of random quadratic forms whose random matrices involved are independent of the random vector. Their Lemma 3.2 states such joint distribution of these random quadratic forms, which can be restated and proved using our Theorem \ref{ksesquilinear}.
\begin{proposition}\label{propan}[\cite{guang}]
Let $S_1=(s_2\cdots s_k)$, independent of $s_1$. Then the random vector
\[
\sqrt{\frac n 2}\left(
\begin{array}{c}
s_1^T(S_1S_1^T)s_1-y_n\int x dG_{y_n}(x)\\
\vdots\\
s_1^T(S_1S_1^T)^is_1-y^i_n\int x^i dG_{y_n}(x)\\[2mm]
s_1^Ts_1-1
\end{array}\right)
\]
is asymptotically normal with mean $0$ and covariance matrix
\[
\left(
\begin{array}{cccc}
B_{11} & \cdots & B_{1i} & B_{1\, i+1}\\
\vdots & \ddots & \vdots &\vdots\\
B_{i1} & \cdots & B_{ii} & B_{i\, i+1}\\
B_{i+1\,1} & \cdots & B_{i+1\,i} & B_{i+1\, i+1}\\
\end{array}
\right)_{(i+1)\times(i+1)}~,\]
where
\begin{align}
&B_{mm}=\left\{
\begin{array}{lllll}
\frac{\nu_4-1}{2}\cdot f^2(m)+f(2m)-f^2(m)~,&&& &1\leq m\leq i~,\\[1mm]
\frac{\nu_4-1}{2}~,&&&&  m=i+1~,
\end{array}\right.\\
&B_{ml}=\left\{
\begin{array}{ll}
\frac{\nu_4-1}{2}\cdot f(m)f(l)+f(m+l)-f(m)f(l)~, & 1\leq m, l\leq i~,\\[1mm]
\frac{\nu_4-1}{2}\cdot f(m)~, & 1 \leq m \leq i, l=i+1~,
\end{array}\right.
\end{align}
here $y_n=k/n$, $y=\lim y_n$, $f(m):=y^m\int x^m dG_y(x)$, and $G_y(x)$ is the limiting spectral distribution of $\frac n k S_1S^T_1$.
\end{proposition}

\noindent The proof of this Proposition is in Section \ref{se4}.

\section{Two applications in spiked population models}\label{application}
It is well known that the empirical spectral distribution of a large-dimensional sample covariance matrix tends to  the Mar\v{c}enko-Pastur distribution $F_y(dx)$:
\[
F_y(dx)=\frac{1}{2\pi xy}\sqrt{(x-a_y)(b_y-x)}dx,~~~~~~~~~~~a_y\leq x\leq b_y~,
\]
where $y=\lim p/n$, $a_y=(1-\sqrt y)^2$ and $b_y=(1+\sqrt y)^2$ under fairly general conditions, see \cite{MP}. Moreover, under a fourth moment assumption, the smallest and largest sample eigenvalues converge almost surely to the end points $a_y$ and $b_y$, respectively.

While in recent empirical data analysis, there is often the case that  some eigenvalues are well separated from the bulk, in order to explain such phenomenon, \cite{John01}  proposed a {\em spiked population model}, where all the population eigenvalues equal to $1$ except some fixed number of them (spikes). Clearly, the spiked population model can be considered as a finite-rank perturbation of the {\em null case} where all the population eigenvalues equal to $1$. Then there raises the question that what's the influence of these spikes on the individual sample eigenvalues.  \cite{BBP05} first unveiled the phase transition phenomenon in the case of complex Gaussian variables, stating that when the population spikes are above (or under) a certain threshold $1+\sqrt y$ (or $1-\sqrt y$), the corresponding extreme sample eigenvalues will jump out of the bulk (become outliers). \cite{Baik06} consider more general random variable: complex or real and not necessarily Gaussian and they found the same transition phenomenon.
As for the central limit theorem, \cite{BBP05} proposed the result for the largest sample eigenvalue in the Gaussian complex case. \cite{Paul07}  found the Gaussian limiting distribution when the population vector is real Gaussian and all the spikes of the population covariance matrix are simple. \cite{BaiYao08}  established the central limit theorem for the largest as well as for the smallest sample eigenvalues under general population variables.

Beyond the sample covariance matrix, there exist many recent and related results concerning the almost sure limit as well as the central limit theorem of the extreme eigenvalue of the Wigner matrix or general Hermitian matrix perturbed by a low rank matrix. Interested reader is referred to \cite{Ca09},  \cite{Benaych11}, \cite{BenaychNadakuditi11}, ~\cite{Ca12},  \cite{AY12}, \cite{RS12} and \cite{Pizzo}, for a selection of such results.

In this section, we establish two new central limit theorems for the extreme sample eigenvalues as well as  sample eigenvector projections. First,  Section \ref{pre} gives  introductions on the model and some preliminary results.  In Section \ref{application11}, a joint central limit theorem is  proposed for groups of packed sample eigenvalues corresponding to the spikes (primary CLT in \cite{BaiYao08} concerns only one such group). Next in Section \ref{application2}, assuming the simple spiked case, we derive a joint CLT for the extreme sample eigenvalue and its corresponding sample eigenvector projection. Such CLT is a new result; indeed, we do not know any CLT related to spike eigenvectors from the literature. Finally, both applications are based on the general CLT for random sesquilinear forms in our Theorem \ref{sesquilinear}.

\subsection{Some notation and preliminary results}\label{pre}

Suppose   the zero-mean complex-valued random vector $x=(\xi^T, \eta^T)^T$, where $\xi=(\xi(1), \cdots, \xi(M))^T$, $\eta=(\eta(1), \cdots, \eta(p))^T$ are independent, of dimension $M$ (fixed) and $p$ ($p \rightarrow \infty$), respectively. And denote $x_i=(\xi_i^T, \eta_i^T)^T$ $(i=1,\cdots, n)$ the $n$ i.i.d. copies of $x$. Moreover, assume that $E|| x||^4 < \infty$ and the coordinates of $\eta$ are independent and identically distributed with unit variance.

The population covariance matrix of the vector $x$ is then
\begin{align}\label{v}
V=\cov(x)=\begin{pmatrix}
           \Sigma & 0 \\
           0 & I_p \\
         \end{pmatrix}~.
\end{align}
Assume $\Sigma$ has the spectral decomposition:
\begin{equation}\label{model}
 \Sigma=U \diag(\underbrace{a_1, \cdots, a_1}_{n_1},
  \cdots,
  \underbrace{a_k, \cdots, a_k}_{n_k})~U^{*},
\end{equation}
where $U$ is an unitary matrix, the $a_i$'s are positive and different from $1$, and the $n_i$'s satisfy $n_1+\cdots +n_k=M$. Besides, let $M_a$ be the number of $j's$ such that $a_j<1-\sqrt y$ (here, $y$ is the limit of dimension to sample size ratio: $y=\lim p/n \in (0,1)$), and let $M_b$ be the number of $j's$ such that $a_j>1+\sqrt y$. More specifically, if we arrange the $a_i's$ in decreasing order, then $\Sigma$ could be diagonalized as
\begin{equation*}
 \diag(\underbrace{\underbrace{a_1, \cdots, a_1}_{n_1},
  \cdots,
  \underbrace{a_{M_b}, \cdots, a_{M_b}}_{n_{M_b}}}_{>1+\sqrt y},
  \cdots,
  \underbrace{\underbrace{a_{k-M_a+1}, \cdots, a_{k-M_a+1}}_{n_{k-M_a+1}},
  \cdots,
  \underbrace{a_k, \cdots, a_k}_{n_k}}_{<1-\sqrt y})~.
\end{equation*}

The sample covariance matrix of $x$ is
\[
S_n=\frac1n\sum_{i=1}^n x_i x_i^{*}~,
\]
which can be partitioned as
\[
S_n=\begin{pmatrix}
      S_{11} & S_{12} \\
      S_{21} & S_{22} \\
    \end{pmatrix}=\begin{pmatrix}
                    X_1X_1^{*} & X_1X_2^{*} \\
                    X_2X_1^{*} & X_2X_2^{*} \\
                  \end{pmatrix}=\frac1n\begin{pmatrix}
                                         \sum \xi_{i}\xi_{i}^{*} & \sum \xi_{i}\eta_{i}^{*} \\
                                         \sum \eta_{i}\xi_{i}^{*} & \sum \eta_{i}\eta_{i}^{*} \\
                                       \end{pmatrix}~,
                                       \]
with
\[
X_1=\frac1{\sqrt{n}}(\xi_1, \cdots, \xi_n)_{M \times n}:=\frac1{\sqrt{n}} \xi_{1:n}~,
\]
\[
X_2=\frac1{\sqrt{n}}(\eta_1, \cdots, \eta_n)_{p \times n}:=\frac1{\sqrt{n}} \eta_{1:n}~.
\]
Since $M$ is fixed and $p\rightarrow \infty$, $n \rightarrow \infty$ such that $p/n \rightarrow y \in (0,1)$, the empirical spectral distribution of the eigenvalues of $S_n$, as well as the one of $S_{22}$, converges to the Mar\v{c}enko-Pastur distribution $F_y(dx)$.
For real constant  $\lambda\notin [a_y,b_y]$,  we define the following integrals  with respect to $F_y(dx)$:
\begin{align}
&m_0(\lambda):=\int \frac{1}{\lambda-x}F_y(dx)~,~~~~~~
m_1(\lambda):=\int \frac{x}{\lambda-x}F_y(dx)~,\nonumber\\
&m_2(\lambda):=\int \frac{x^2}{(\lambda-x)^2}F_y(dx)~,~~
m_3(\lambda):=\int \frac{x}{(\lambda-x)^2}F_y(dx)~,\nonumber\\
&m_4(\lambda):=\int \frac{1}{(\lambda-x)^2}F_y(dx)~,~~
m_5(\lambda):=\int \frac{x}{(\lambda-x)^3}F_y(dx)~,\nonumber\\
&m_6(\lambda):=\int \frac{x^2}{(\lambda-x)^4}F_y(dx)~,~~
m_7(\lambda):=\int \frac{x^2}{(\lambda-x)^3}F_y(dx)~.\label{mi}
\end{align}

Let $l_1\geq l_2\geq \cdots \geq l_p$ be the eigenvalues of $S_n$.
Let $s_j=n_1+\cdots +n_j$ for $1\leq j\leq M_b$ or $k-M_a+1 \leq j \leq k$. \cite{Baik06} derive the almost sure limit of those extreme sample eigenvalues. They have proven that for each $m \in \{1,\cdots, M_b\}$ or $m \in \{k-M_a+1,\cdots, k\}$ and $s_{m-1}<i\leq s_m$,
\[
l_i\rightarrow \lambda_m=\phi(a_m):= a_m+\frac{y a_m}{a_m-1}~
\]
almost surely. In other words, if a spike eigenvalue $a_m$ lies outside the interval $[1-\sqrt y, 1+\sqrt y]$, then the $n_m$-packed sample eigenvalues $\{l_i, i \in J_m\}$ (associated to $a_m$) converge to the limit $\lambda_m$, which is outside the support of the M-P distribution $[a_y,b_y]$ (here, we denote $J_m=(s_{m-1},s_m]$ when $m \in \{1,\cdots, M_b\}$ or $m \in \{k-M_a+1,\cdots, k\}$).

Recently  \cite{BaiYao08} derives the CLT for those extreme sample eigenvalues. More specifically, let $\delta_{n,i}:=\sqrt n(l_i-\lambda_m)$, where $m \in \{1,\cdots, M_b\}$ or $m \in \{k-M_a+1,\cdots, k\}$, $i \in J_m$, and $\lambda_m=\phi(a_m)\notin [a_y,b_y]$ as defined before.
They have proven that $\delta_{n,i}$ tends to the solution $v$ of the following equation:
\begin{eqnarray}\label{solution}
\Big|-\big[U^{*}R_n(\lambda_m)U\big]_{mm}+v\big(1+ym_3(\lambda_m)a_m\big)I_{n_m}+o_n(1)\Big|=0~,
\end{eqnarray}
here $|\cdot|$ stands for determinant, $\big[U^{*}R_n(\lambda_m)U\big]_{mm}$ is the $m$-th diagonal block of $U^{*}R_n(\lambda_m)U$ corresponding to the index $\{u,v \in J_m\}$, and
\begin{eqnarray}
&&R_n(\lambda)=\frac{1}{\sqrt{n}}\Big\{\xi_{1:n}\big(I+A_n(\lambda)\big)\xi_{1:n}^{*}-\Sigma tr\big(I+A_n(\lambda)\big)\Big\}~,\nonumber\\
&&A_n(\lambda)=X_2^{*}(\lambda I-X_2X_2^{*})^{-1}X_2~\nonumber.
\end{eqnarray}
 Let $R(\lambda)$ denote the $M \times M$ matrix limit of  $R_n(\lambda)$, and $\widetilde{R}(\lambda):= U^{*}R(\lambda)U$. According to \eqref{solution}, it says that $\delta_{n,i}$ tends to an eigenvalue of the matrix $(1+ym_3(\lambda_m)a_m)^{-1}[\widetilde{R}(\lambda_m)]_{mm}$. Besides, since the index $i$ is arbitrary over $J_m$, all the $J_{m}$ random variables $\sqrt{n}\{l_i-\lambda_m,~i\in J_m\}$ converge almost surely to the set of eigenvalues of this matrix. The following theorem in \cite{BaiYao08} identifies the covariance of the elements within the limit matrix $R(\lambda)$. For simplicity, we only consider the real case in all the following unless otherwise noted.

\begin{proposition}\label{ori}[\cite{BaiYao08}]
Assume that the variables $\xi$ and $\eta$ are real, then the random matrix $R=R_{ij}$ is symmetric, with zero-mean Gaussian entries, having the following covariance function: for $1\leq i\leq j\leq M$ and $1\leq i^{'}\leq j^{'}\leq M$
\begin{align*}\label{yao}
&~~~~\cov~(R(i, j), R(i^{'}, j^{'}))\nonumber\\
&=w\Big\{E[\xi(i)\xi(j)\xi(i^{'})\xi(j^{'})]-\Sigma_{ij}\Sigma_{i^{'}j^{'}}\Big\}
+\big(\theta-w\big)
\Sigma_{ij^{'}}\Sigma_{i^{'}j}\nonumber\\
&~~~+\big(\theta-w\big)\Sigma_{ii^{'}}\Sigma_{jj^{'}}~,
\end{align*}
where the constants $\theta$ and $w$ are defined as follows:
\begin{eqnarray*}
&&\theta=1+2ym_1(\lambda)+ym_2(\lambda)~,\\
&&w=1+2ym_1(\lambda)+\bigg(\frac{y(1+m_1(\lambda))}{\lambda-y(1+m_1(\lambda))}\bigg)^2~.
\end{eqnarray*}
\end{proposition}

\subsection{Application 1: Asymptotic joint distribution of two groups of extreme sample eigenvalues in the spiked population model}\label{application11}
 In this subsection, we consider the asymptotic joint distribution of two groups of extreme sample eigenvalues, say, $\{l_i,~i \in J_m\}$ and $\{l_{i^{'}},~i^{'} \in J_{m^{'}}\}$ ($m\neq m^{'}$) when $\Sigma$ has the structure \eqref{model}, namely the random vector
   \begin{equation*}
\left(
\begin{array}{c}
 \{\sqrt{n}(l_{i}-\lambda_m),~i \in J_m\}\\
 \{\sqrt{n}(l_{i^{'}}-\lambda_{m^{'}}),~i^{'} \in J_{m^{'}}\}\\
 \end{array}
\right)~.
\end{equation*} Following the work of \cite{BaiYao08}, we know that this  $n_m+n_{m^{'}}$ dimensional random
vector
converges to the eigenvalues of the symmetric $(n_m+n_{m^{'}})\times (n_m+n_{m^{'}})$ random matrix
\begin{equation}\label{block}
\begin{pmatrix}
 \frac{[\widetilde{R}(\lambda_m)]_{mm}}{1+ym_3\left(\lambda_m\right)a_m}&0\\[2mm]
 0&\frac{[\widetilde{R}(\lambda_{m^{'}})]_{m^{'}m^{'}}}{1+ym_3\left(\lambda_{m^{'}}\right)a_{m^{'}}}\\
 \end{pmatrix}~.
\end{equation}
Here, this random matrix \eqref{block} has two diagonal blocks with dimension $n_m$ and $n_{m^{'}}$, respectively. The covariance function of the elements within each block has been fully identified by \cite{BaiYao08}, see Proposition \ref{ori}. But if we consider them as a whole, there's still need to explore the covariance between the elements from the different two blocks $[\widetilde{R}(\lambda_m)]_{mm}$ and $[\widetilde{R}(\lambda_{m^{'}})]_{m^{'}m^{'}}$.

We establish such a covariance function in Theorem \ref{joint}  when the observation vector $x$ is real with the help of our Corollary \ref{real}. However, it can also be generalized to the complex case by considering the real and imaginary parts as two independent real random variables with the help of our Theorem \ref{sesquilinear}, readers who are interested in this can  refer to \cite{BaiYao08} (see the proof of their Proposition 3.2).

\subsubsection{Main result}
\begin{theorem}\label{joint}
Assume that the variables $\xi$ and $\eta$ are real, then the two diagonal blocks of the $2M\times 2M$ random matrix
\begin{equation}\label{ourblock}
\left(
\begin{array}{cc}
 R(\lambda_m)&0\\
 0&R(\lambda_{m^{'}})\\
 \end{array}
\right)~
\end{equation}
are symmetric, having zero-mean Gaussian entries, with the following covariance function between each other: for $1\leq i\leq j \leq M$ and $1\leq i^{'}\leq j^{'} \leq M$, we have
\begin{align}
&~~~~\cov~(R(\lambda_m)(i, j), R(\lambda_{m^{'}})(i^{'}, j^{'}))\nonumber\\
&=w(m,m^{'})\Big\{\E[\xi(i)\xi(j)\xi(i^{'})\xi(j^{'})]-\Sigma_{ij}\Sigma_{i^{'}j^{'}}\Big\}\nonumber\\
&~~+\big(\theta(m,m^{'})-w(m,m^{'})\big)\Sigma_{ij^{'}}\Sigma_{i^{'}j}\nonumber\\
&~~+\big(\theta(m,m^{'})-w(m,m^{'})\big)\Sigma_{ii^{'}}\Sigma_{jj^{'}}~,\label{12}
\end{align}
where
\[
\theta(m,m^{'})=1+ym_1(\lambda_m)+ym_1(\lambda_{m^{'}})+y\Bigg(\frac{\lambda_{m^{'}}}{\lambda_m-\lambda_{m^{'}}}m_1(\lambda_{m^{'}})+\frac{\lambda_m}{\lambda_{m^{'}}-\lambda_m}m_1(\lambda_m)\Bigg)~,
\]
\[
w(m,m^{'})=1+ym_1(\lambda_m)+ym_1(\lambda_{m^{'}})+\frac{y^2(1+m_1(\lambda_m))(1+m_1(\lambda_{m^{'}}))}{(\lambda_m-y(1+m_1(\lambda_m)))(\lambda_{m^{'}}-y(1+m_1(\lambda_{m^{'}})))}~.
\]
\end{theorem}

\begin{remark}
If we restrict the index $(i,j)$ to the region $J_m \times J_m$ and $(i^{'},j^{'})$ to $J_{m^{'}}\times J_{m^{'}}$, we can get the covariance function between the two blocks of \eqref{block}. And it should be noticed that the two regions $J_m \times J_m$ and $J_{m^{'}}\times J_{m^{'}}$ do not intersect with each other.
\end{remark}

\begin{remark}
In general, the covariance of the elements from  two blocks are not independent asymptotically, that is $\cov~(R(\lambda_m)(i, j), R(\lambda_{m^{'}})(i^{'}, j^{'}))\neq 0$. Notice that  same phenomenon also exists in the Wigner case, for example, see Theorem 2.11 in \cite{AY12}.
\end{remark}

\begin{remark}\label{rem1}
If the coordinates $\{\xi(i)\}$ of $\xi$ are independent (thus, $\Sigma$ is diagonal and $U=I_M$), \cite{BaiYao08} has already proved that the covariance matrix within each diagonal block in \eqref{ourblock} is diagonal; in other words, the Gaussian matrix $R(\lambda_m)$ and $R(\lambda_{m^{'}})$  are both made with independent entries. And by noting that the regions $J_m \times J_m$ and $J_{m^{'}} \times J_{m^{'}}$ are disjoint, the only covariance function that may exist between the two blocks is $\cov(R(\lambda_m)(i,i),R(\lambda_{m^{'}})(i^{'},i^{'}))$ ($i \in J_m,~i^{'} \in J_{m^{'}}$). Using \eqref{12} and the fact that $\{\xi(i)\}$ are independent, we have
\begin{align*}
&~~~~\cov~(R(\lambda_m)(i, i), R(\lambda_{m^{'}})(i^{'}, i^{'}))\\
&=w(m,m^{'})\Big\{\E[\xi(i)^2\xi(i^{'})^2]-\Sigma_{ii}\Sigma_{i^{'}i^{'}}\Big\}
+2\big(\theta(m,m^{'})-w(m,m^{'})\big)\big(\Sigma_{ii^{'}}\big)^2\\
&=w(m,m^{'})\Big\{\Sigma_{ii}\Sigma_{i^{'}i^{'}}-\Sigma_{ii}\Sigma_{i^{'}i^{'}}\Big\}
+2\big(\theta(m,m^{'})-w(m,m^{'})\big)\big(\Sigma_{ii^{'}}\big)^2\\
&=0~,
\end{align*}
which means that the two diagonal blocks in \eqref{block} are independent.
Besides, \cite{BaiYao08} have already pointed out the  variances within each block:
\begin{eqnarray}
&&Var (R(i,j))=\theta \Sigma_{ii}\Sigma_{jj}, ~~~~i<j\label{vij} \\
&&Var (R(i,i))=w(E \xi(i)^4-3\Sigma_{ii}^2)+2\theta\Sigma_{ii}^2 \label{vii}~.
\end{eqnarray}
Therefore, if $\{\xi(i)\}$ are independent, then any two groups of packed extreme sample eigenvalues $\{\sqrt n(l_{i}-\lambda_m),i \in J_m\}$ and $\{\sqrt n(l_{i^{'}}-\lambda_{m^{'}}),i^{'} \in J_{m^{'}}\}$ are asymptotically independent, converging to the eigenvalues of the Gaussian random matrices   $\frac{1}{1+ym_3(\lambda_m)a_m}[R(\lambda_m)]_{mm}$ and $\frac{1}{1+ym_3(\lambda_{m^{'}})a_{m^{'}}}[R(\lambda_{m^{'}})]_{m^{'}m^{'}}$, respectively. And both the Gaussian random matrices are made with independent entries, with a fully identified variance  function  given by \eqref{vij} and \eqref{vii}.
Moveover, if the observations are Gaussian, \eqref{vii} reduces to $Var (R(i,i))=2\theta\Sigma_{ii}^2$.

\end{remark}

\subsubsection{Conditions that two groups of packed extreme sample eigenvalues are pairwise independent}
An interesting question in the asymptotical analysis of spiked eigenvalues is to know whether two groups of packed extreme sample eigenvalues are asymptotically pairwise independent. In Remark \ref{rem1}, we have seen that when  $\{\xi(i)\}$ are independent, $\{\sqrt n(l_{i}-\lambda_m),i \in J_m\}$ and $\{\sqrt n(l_{i^{'}}-\lambda_{m^{'}}),i^{'} \in J_{m^{'}}\}$ are asymptotically  independent.

We  aim to relax the independent restriction of $\{\xi(i)\}$ under the condition that all the eigenvalues of $\Sigma$ are simple, that is, $\Sigma$ has the spectral decomposition:
\begin{equation*}
\Sigma=U\left(
\begin{array}{cccc}
 a_1 & 0 &\cdots &0\\
 0 & a_2 &\cdots &0\\
 \vdots & \vdots & \vdots & \vdots \\
 0&0&\cdots &a_M
 \end{array}
\right)U^{*}~,
\end{equation*}
where the $a_i's$ are arranged in decreasing order. We discuss the condition that when the extreme sample eigenvalues are pairwise independent, asymptotically.

Let $l_{i}$, $l_{j}$ denote the extreme sample eigenvalues correspond to two different spikes $a_i$ and $a_j$, where $a_i,a_j \notin [1-\sqrt y,1+\sqrt y]$.
Then, the two-dimensional random vector
 \begin{equation*}
 \begin{pmatrix}
   \delta_{n,i} \\[1mm]
   \delta_{n,j} \\
 \end{pmatrix}=
 \begin{pmatrix}
 \sqrt{n}\big(l_{i}-\lambda_{i}\big)\\[1mm]
 \sqrt{n}\big(l_{j}-\lambda_{j}\big)\\
 \end{pmatrix}
\end{equation*}
converges to the eigenvalues of the following random matrix:
\begin{equation*}
 \begin{pmatrix}
 \frac{1}{1+ym_3(\lambda_{i})a_i}[\widetilde{R}(\lambda_{i})]_{ii}&0\\[1mm]
 0&\frac{1}{1+ym_3(\lambda_{j})a_j}[\widetilde{R}(\lambda_{j})]_{jj}\\
 \end{pmatrix}~.
\end{equation*}
Since all the eigenvalues of $\Sigma$ are simple, the multiplicity numbers $n_i$ and $n_j$ both equal to $1$. Therefore, $[\widetilde{R}(\lambda_{i})]_{ii}$ and $[\widetilde{R}(\lambda_{j})]_{jj}$ are now two Gaussian random variables (actually, they are the $(i,i)$-th and $(j,j)$-th elements of the $M \times M$ Gaussian random matrices $\widetilde{R}(\lambda_{i})$ and $\widetilde{R}(\lambda_{j})$, respectively, denoted as $\widetilde{R}(\lambda_{i})(i,i)$ and $\widetilde{R}(\lambda_{j})(j,j)$).
As a result,
\[
\begin{pmatrix}
  \delta_{n,i} & \delta_{n,j} \\
\end{pmatrix}^T
\]
actually converges to the Gaussian random vector
 \begin{equation*}
\left(
\begin{array}{c}
 \frac{1}{1+ym_3(\lambda_{i})a_i}\widetilde{R}(\lambda_{i})(i,i)\\[1mm]
 \frac{1}{1+ym_3(\lambda_{j})a_j}\widetilde{R}(\lambda_{j})(j,j)\\
 \end{array}
\right)
\end{equation*}
with
\begin{align}
&Var~\big(R(\lambda_{i})(i,i)\big)=w(i)\big\{\E\big[\xi(i)^4\big]-\Sigma_{ii}^2\big\}+2\big(\theta(i)-w(i)\big)\Sigma_{ii}^2~,\label{va1}\\
&Var~\big(R(\lambda_{j})(j,j)\big)=w(j)\big\{\E\big[\xi(j)^4\big]-\Sigma_{jj}^2\big\}+2\big(\theta(j)-w(j)\big)\Sigma_{jj}^2~,\label{va2}\\
&\cov~\big(R(\lambda_{i})(i,i), R(\lambda_{j})(j,j)\big)=w(i,j)\big\{\E\big[\xi(i)^2\xi(j)^2\big]-\Sigma_{ii}\Sigma_{jj}\big\} \nonumber\\
&~~~~~~~~\quad\quad ~~~~~~~~\quad\quad ~~~~~~~~~~~+2\big(\theta(i,j)-w(i,j)\big)\Sigma_{ij}^2~,\label{co}
\end{align}
where $$\theta(i)=1+2ym_1(\lambda_i)+ym_2(\lambda_i)~,$$ $$w(i)=1+2ym_1(\lambda_i)+\left(\frac{y(1+m_1(\lambda_i))}{\lambda_i-y(1+m_1(\lambda_i))}\right)^2$$ are given in \cite{BaiYao08}.
From the definitions of $w(i,j)$ and $\theta(i,j)$ in Theorem \ref{joint}, taking the fact that $m_1(\lambda_{i})=1/(a_i-1)$ (see Lemma \ref{value}) into consideration, we have,
\begin{eqnarray*}
w(i,j)&=&1+ym_1(\lambda_{i})+ym_1(\lambda_{j})\\
&&+\frac{y^2\big(1+m_1(\lambda_{i})\big)\big(1+m_1(\lambda_{j})\big)}{\big(\lambda_{i}-y(1+m_1(\lambda_{i}))\big)\big(\lambda_{j}-y(1+m_1(\lambda_{j}))\big)}\\
&=&1+\frac{y}{a_i-1}+\frac{y}{a_j-1}+\frac{y^2}{(a_i-1)(a_j-1)}\\
&=&\frac{(y+a_i-1)(y+a_j-1)}{(a_i-1)(a_j-1)}~,\\
\end{eqnarray*}
\begin{eqnarray*}
\theta(i,j)-w(i,j)&=&y\Bigg(\frac{\lambda_{j}}{\lambda_{i}-\lambda_{j}}m_1(\lambda_{j})+\frac{\lambda_{i}}{\lambda_{j}-\lambda_{i}}m_1(\lambda_{i})\Bigg)\\
&&-\frac{y^2}{(a_i-1)(a_j-1)}\\
&=&y\cdot\frac{(y+a_i-1)(y+a_j-1)}{(a_i-1)(a_j-1)\big[(a_i-1)(a_j-1)-y\big]}~.\\
\end{eqnarray*}
The values of $w(i,j)$ will always be positive whenever $a_i, a_j\notin [1-\sqrt{y},1+\sqrt{y}]$, while $\theta(i,j)-w(i,j)$ will be negative if $a_i>1+\sqrt{y}$ and $0<a_j<1-\sqrt{y}$ (corresponding to  one extreme large and one extreme small sample eigenvalues), and positive if $a_i, a_j>1+\sqrt{y}$ or $0<a_i, a_j<1-\sqrt{y}$ (corresponding to  two extreme large or two extreme small sample eigenvalues).

Therefore, if any two extreme large (or small) sample eigenvalues are mutually independent (equivalent to the condition that $\cov(R(\lambda_i)(i,i),R(\lambda_j)(j,j))=0$), a sufficient and necessary condition is
\[\E\big[\xi(i)^2\xi(j)^2\big]-\Sigma_{ii}\Sigma_{jj}=0~,\] and
\[\E[\xi(i)\xi(j)]=0~(=\E\xi(i)\E\xi(j))~;\] another way of saying this is
 \begin{equation*}
(*)~~~~\left\{
\begin{array}{l}
 \cov\big(\xi(i), \xi(j)\big)=0 ~~~~(\Sigma~ is~ diagonal~or~U=I_M)~, and\\
 \cov\big(\xi(i)^2, \xi(j)^2\big)=0\\
 \end{array}
 \right.~.
\end{equation*}
Obviously, when $\{\xi(i)\}$ are independent, the condition $(*)$ is satisfied.

We consider  a special case that the observations are Gaussian, with a diagonal population covariance matrix. This model satisfies condition $(*)$. It is due to the fact that when the observations are Gaussian, uncorrelation between $\xi(i)$ and $\xi(j)$ implies independence, which further implies $\xi(i)^2$ and $\xi(j)^2$ are uncorrelated. Therefore, if the observations are Gaussian and the population covariance matrix is diagonal, then any two extreme large (or small) sample eigenvalues are mutually independent. Furthermore, we can derive explicitly the joint distribution of $\delta_{n,i}$ and $\delta_{n,j}$.
According to \eqref{va1}, \eqref{va2} and \eqref{co}, we have a much more simplified form due to the Gaussian assumption:
\begin{eqnarray}
&&Var~\big(R(\lambda_{i})(i,i)\big)=2\theta(i)a_i^2~,\nonumber\\
&&Var~\big(R(\lambda_{j})(j,j)\big)=2\theta(j)a_j^2~,\nonumber\\
&&\cov~\big(R(\lambda_{i})(i,i), R(\lambda_{j})(j,j)\big)=0\label{gauco}~,
\end{eqnarray}
where $\theta(i)=\frac{(a_i-1+y)^2}{(a_i-1)^2-y}$ and $\theta(j)=\frac{(a_j-1+y)^2}{(a_j-1)^2-y}$ by definition.
And using the expression $m_3(\lambda)=\frac{1}{(a-1)^2-y}$ (see Lemma \ref{value}), we finally derive the asymptotic joint distribution:
\[
\begin{pmatrix}
 \sqrt{n}\big(l_{i}-\lambda_{i}\big)\\[2mm]
 \sqrt{n}\big(l_{j}-\lambda_{j}\big)\\
\end{pmatrix}
 \Longrightarrow
\mathcal{N}\left(
        \begin{pmatrix}
          0 \\[2mm]
          0 \\
        \end{pmatrix},
        \begin{pmatrix}
          \frac{2a_i^2[(a_i-1)^2-y]}{(a_i-1)^2} & 0 \\
          0 & \frac{2a_j^2[(a_j-1)^2-y]}{(a_j-1)^2} \\
        \end{pmatrix}
\right)~.
\]

But, if we only assume $\Sigma$ is diagonal, and no Gaussian assumptions are made, things are different.  One such example is that $\xi(i)$ and $\xi(j)$ come from the uniform distribution inside the ellipse:
\begin{eqnarray*}
\frac{\xi(i)^2}{16}+\frac{\xi(j)^2}{36}\leq 1 ,
\end{eqnarray*}
one can check that $\E\xi(i)\xi(j)=\E\xi(i)\cdot \E\xi(j)=0$, but $\E\xi(i)^2=4$, $\E\xi(j)^2=9$ and $\E\xi(i)^2\xi(j)^2=24$, that is $\E\xi(i)^2\xi(j)^2\neq \E\xi(i)^2\cdot \E\xi(j)^2$, therefore, condition $(*)$ is not satisfied. From this example, we see there could happen that although $\xi(i)$ and $\xi(j)$ are uncorrelated, $\xi(i)^2$ and $\xi(j)^2$ are correlated. And in such a case, even though the population covariance matrix is diagonal, the two extreme large (or small) eigenvalues of the sample covariance matrix may actually have correlation between each other.

A small simulation is conducted below to check this covariance formula according to the two cases mentioned above. The dimension $p$ is fixed to be $200$ and the sample size $n$ is fixed to be $300$. We choose two spikes $a_1=9$ and $a_2=4$, which are both larger than the critical value $1+\sqrt{y}$ ($=1+\sqrt{2/3}$). We repeat $10000$ times to calculate the empirical covariance value between the largest ($l_{1}$) and the second largest ($l_{2}$) sample eigenvalues. The first case is the two-dimensional multivariate Gaussian vector $(\xi(1),\xi(2))^{T}$,
which has a joint distribution
\[
\mathcal{N}\left(
        \begin{pmatrix}
          0 \\
          0 \\
        \end{pmatrix},
      \begin{pmatrix}
          9 & 0 \\
          0 & 4 \\
        \end{pmatrix}
\right)~.
\]
According to \eqref{gauco}, the theoretical covariance value between $l_{1}$ and $l_{2}$ should be 0, and the empirical covariance value from the $10000$ sample simulated turns out to be $0.0019$.
The second case is the aforementioned  uniform distribution  inside the ellipse:
$\xi(1)^2/36+\xi(2)^2/16\leq 1$. This time, the theoretical covariance value between $l_{1}$ and $l_{2}$ could  be calculated as $-0.0366$ according to \eqref{co}, and the empirical covariance value from the $10000$ sample simulated turns out to be $-0.0371$.  The two errors are both smaller than the order $O(1/\sqrt{10000})$ under both cases.

\subsection{Application 2: Asymptotic joint distribution of the largest sample eigenvalue and its corresponding sample eigenvector projection}\label{application2}
In this subsection, we consider the joint central limit theorem of extreme sample eigenvalue and its corresponding sample eigenvector projection, which may find  applications in principal component scores, where both the eigenvalue and its eigenvector are involved, see \cite{S12}.

Let the population covariance matrix  be diagonal with $k$ simple spikes:
\begin{equation*}\label{model2}
 V=\diag(\underbrace{a_1, \cdots, a_k}_k, \underbrace{1, \cdots, 1}_{p})~,
\end{equation*}
where now the $\Sigma$ in \eqref{v} reduces to a diagonal matrix $\diag(a_1, \cdots, a_k)$ with all the diagonal elements $a_i$ larger than the critical value $1+\sqrt y$. The sample covariance matrix $S_n$ is also partitioned as before:
\[
S_n=\begin{pmatrix}
      S_{11} & S_{12} \\
      S_{21} & S_{22} \\
    \end{pmatrix}=\begin{pmatrix}
                    X_1X_1^{*} & X_1X_2^{*} \\
                    X_2X_1^{*} & X_2X_2^{*} \\
                  \end{pmatrix}=\frac1n\begin{pmatrix}
                                         \sum \xi_{i}\xi_{i}^{*} & \sum \xi_{i}\eta_{i}^{*} \\
                                         \sum \eta_{i}\xi_{i}^{*} & \sum \eta_{i}\eta_{i}^{*} \\
                                       \end{pmatrix}~,
                                       \]
with
\[
X_1=\frac1{\sqrt{n}}(\xi_1, \cdots, \xi_n)_{k \times n}:=\frac1{\sqrt{n}} \xi_{1:n}~,
\]
\[
X_2=\frac1{\sqrt{n}}(\eta_1, \cdots, \eta_n)_{p \times n}:=\frac1{\sqrt{n}} \eta_{1:n}~,
\]
which are mutually independent. And we denote $\nu_4(i)=\E \xi(i)^4/a_i^2-3$ for $i=1, \cdots, k$ as the kurtosis coefficient of the $i$-th coordinate of $\xi$.

Now suppose $l_i$ is an extreme eigenvalue of $S_n$, converging to the value $\lambda_i=\phi(a_i)=a_i+ya_i/(a_i-1)$ and let $(u_i,v_i)^T$
be the corresponding sample eigenvector with $u_i$ its first $k$ components  and $v_i$ the remaining $p$ components. We derive the following central limit theorem that establishes the asymptotic joint distribution of  the extreme sample eigenvalue $l_i$ and its corresponding sample eigenvector projection $u_i(i)^2$ (here $u_i(i)$ stands for the $i$-th element of the $k \times 1$ vector $u_i$). Notice that the population eigenvector corresponding to the spike $a_i$ is simply $e_i=(0, \cdots, 1, \cdots, 0)^{T}$, the $i$-th standard canonical basis vector. Therefore, $u_i(i)$ represents the inner product between the sample eigenvector $(u_i,v_i)^T$ and the population one $e_i$.

\begin{theorem}\label{lu}
\begin{align*}
\begin{pmatrix}
  \sqrt n\left(u_i(i)^{2}-\frac{(a_i-1)^2-y}{(a_i-1)(a_i-1+y)}\right) \\[2mm]
  \sqrt n (l_i-\lambda_i) \\
\end{pmatrix}\Longrightarrow \N \left(\begin{pmatrix}
                                        0 \\[2mm]
                                        0 \\
                                      \end{pmatrix},\begin{pmatrix}
                                                      v_{11} & v_{12} \\[2mm]
                                                      v_{12} & v_{22} \\
                                                    \end{pmatrix}\right)~,
\end{align*}
where
\begin{align*}
&v_{11}=\frac{a_i^2y^2(a_i^2+y-1)^2}{(a_i-1)^4(a_i-1+y)^4}\nu_4(i)+\frac{2a_i^2y((a_i+y-1)^2+ya_i^2)}{((a_i-1)^2-y)(a_i-1+y)^4}~,\\
&v_{12}=\frac{ya_i^2(a_i^2-1+y)((a_i-1)^2-y)}{(a_i-1)^6(a_i-1+y)^4}\nu_4(i)+\frac{2a_i^3y}{(a_i-1)(a_i-1+y)^2}~,\\
&v_{22}=\frac{a_i^2((a_i-1)^2-y)^2}{(a_i-1)^4}\nu_4(i)+\frac{2a_i^2((a_i-1)^2-y)}{(a_i-1)^2}~.
\end{align*}
\end{theorem}

\begin{remark}
If the observations are Gaussian ($\nu_4(i)=0$ for $i=1, \cdots, k$), then the three values above are simplified to be:
\begin{align*}
&v_{11}=\frac{2a_i^2y((a_i+y-1)^2+ya_i^2)}{((a_i-1)^2-y)(a_i-1+y)^4}~,\\
&v_{12}=\frac{2a_i^3y}{(a_i-1)(a_i-1+y)^2}~,\\
&v_{22}=\frac{2a_i^2((a_i-1)^2-y)}{(a_i-1)^2}~.
\end{align*}
\end{remark}

\begin{remark}
Trivially, the following central limit theorem of the eigenvector projection holds
\[
\sqrt n\left(u_i(i)^{2}-\frac{(a_i-1)^2-y}{(a_i-1)(a_i-1+y)}\right)\longrightarrow \mathcal{N}(0,v_{11})~.
\]
In particular,
\[
u_i(i)^{2}\xrightarrow{p} \frac{(a_i-1)^2-y}{(a_i-1)(a_i-1+y)}~.
\]
Observe that this limit $\in (0,1)$. In particular, the sample eigenvector  does not converge to the population eigenvector; only their angle tends to a limit. Notice that the limit of the angle has already been established by \cite{Paul07} for the Gaussian case and \cite{BenaychNadakuditi11} on somewhat different but closely related random matrix models with a finite-rank perturbation.
\end{remark}

\section{Proof of Proposition \ref{propan}, Theorem \ref{joint} and \ref{lu}} \label{proof}
\subsection{Proof of Proposition \ref{propan}}\label{se4}
\begin{proof}
We give a short proof of  Proposition \ref{propan} using our Theorem \ref{ksesquilinear}.
 Let
\begin{align*}
&U^{(m)}=\frac{1}{\sqrt n} \left[s_1^T(S_1S^T_1)^m s_1-\frac 1n \tr (S_1S^T_1)^m\right]~, ~1\leq m \leq i~,\\
&U^{(i+1)}=\frac{1}{\sqrt n}\left[s_1^Ts_1-1\right]~,
\end{align*}
so we have
\begin{align*}
&k=i+1~,\\
&X(1)=s_1:=\frac{1}{\sqrt n}(v_{11}, \cdots, v_{1n})^T=Y(1)~,~(K=1)~,\\
&\rho(l)=\E x_{11}y_{11}=\frac 1n~,\\
&A_1=\frac {1}{n^2}(\nu_4-1),~ A_2=A_3=\frac{1}{n^2}~,
\end{align*}
and the matrix
$$A^{(m)}:=\left\{\begin{array}{ll}
(S_1S^T_1)^m~, & 1\leq m\leq i\\
I_n~, & m=i+1
\end{array}\right.~.$$
Then,
\begin{align*}
\tau_m=\theta_m=\limn \frac 1n \tr[A^{(m)}]^2&=\left\{\begin{array}{ll}
\limn \frac{1}{n} \tr (S_1S^T_1)^{2m}~, & 1\leq m \leq i\\
1~,& m=i+1
\end{array}\right.\\
&=\left\{\begin{array}{lll}
y^{2m}\int x^{2m}dG_y(x)~,& & 1\leq m \leq i\\
1~,& & m=i+1
\end{array}\right.~,
\end{align*}

\begin{align*}
\tau_{mm^{'}}=\theta_{mm^{'}}&=\left\{\begin{array}{ll}
\limn \frac{1}{n} \tr \left[(S_1S^T_1)^{m}(S_1S^T_1)^{m^{'}}\right]~, & 1\leq m, m^{'} \leq i~,\\[1mm]
\limn \frac{1}{n} \tr (S_1S^T_1)^{m^{'}}~,& m=i+1, 1\leq m^{'} \leq i~,
\end{array}\right.\\
&=\left\{\begin{array}{llllll}
y^{m+m^{'}}\int x^{m+m^{'}}dG_y(x)~, &&&&&1\leq m, m^{'}\leq i~,\\
y^{m^{'}}\int x^{m^{'}}dG_y(x)~, &&&&& m=i+1,1\leq m^{'}\leq i~,
\end{array}\right.
\end{align*}

\begin{align*}
w_m=\limn\frac{1}{n}\tr [A^{(m)}\circ A^{(m)}]&=\left\{\begin{array}{ll}
(\limn \frac{1}{n} \tr (S_1S^T_1)^{m})^2~, & 1\leq m \leq i~,\\[1mm]
(\limn \frac{1}{n} \tr I_n)^2~, & m=i+1~,
\end{array}\right.\\
&=\left\{\begin{array}{lll}
y^{2m}(\int x^{m}dG_y(x))^2~, & &1\leq m \leq i~,\\
1~, && m=i+1~,
\end{array}\right.
\end{align*}

\begin{align*}
w_{mm^{'}}&=\limn\frac{1}{n}\tr [A^{(m)}\circ A^{(m^{'})}]\\
&=\left\{\begin{array}{ll}
\limn \frac{1}{n} \tr (S_1S^T_1)^{m}\cdot \lim_{n\rightarrow \infty} \frac{1}{n} \tr (S_1S^T_1)^{m^{'}}~, & 1\leq m, m^{'} \leq i~,\\[1mm]
\limn \frac{1}{n} \tr (S_1S^T_1)^{m}~, & m^{'}=i+1~, 1\leq m \leq i~,
\end{array}\right.\\
&=\left\{\begin{array}{llll}
y^{m}\int x^{m}dG_y(x)\cdot y^{m^{'}}\int x^{m^{'}}dG_y(x)~, &&&1\leq m, m^{'} \leq i~,\\[1mm]
y^{m}\int x^{m}dG_y(x)~, &&& m^{'}=i+1,  1\leq m \leq i~,
\end{array}\right.
\end{align*}
Combine all that above with \eqref{bii} and \eqref{bij} leads to the result.

\end{proof}

\subsection{Proof of Theorem \ref{joint}}
\begin{proof}
We prove this result with the help of  Corollary \ref{real}. Consider
\begin{equation*}
\left(
\begin{array}{c}
 \frac{1}{\sqrt{n}}u(i)(I+A_n(\lambda_m))u(j)^{T} \\
 \frac{1}{\sqrt{n}}u(i^{'})(I+A_n(\lambda_{m^{'}}))u(j^{'})^{T}\\
 \end{array}
\right)_{1 \leq i \leq j \leq M, ~1 \leq i^{'} \leq j^{'} \leq M }~,
\end{equation*}
with $u(i)=\big(\xi_1(i), \cdots, \xi_n(i)\big)$. Moreover, we define $X(l)=u(i)^{T}, Y(l)=u(j)^{T}$, $X(l^{'})=u(i^{'})^{T}, Y(l^{'})=u(j^{'})^{T}$, with $l=(i, j)$, $l^{'}=(i^{'}, j^{'})$, where $l$ and $l^{'}$ both have $K=\frac{M(M+1)}{2}$ options. Recall the definition of $R_n$, we have:
\begin{eqnarray*}
&&R_n(\lambda_m)=\frac{1}{\sqrt{n}}\Big\{\xi_{1:n}(I+A_n(\lambda_m))\xi^{*}_{1:n}-\Sigma \tr(I+A_n(\lambda_m))\Big\}~,\\
&&R_n(\lambda_{m^{'}})=\frac{1}{\sqrt{n}}\Big\{\xi_{1:n}(I+A_n(\lambda_{m^{'}}))\xi^{*}_{1:n}-\Sigma \tr(I+A_n(\lambda_{m^{'}}))\Big\}~.
\end{eqnarray*}
By applying Corollary \ref{real}, we have
\[
\cov(R(\lambda_m)(i, j), R(\lambda_{m^{'}})(i^{'}, j^{'}))=w_3A_1+(\tau_3-w_3)A_2+(\theta_3-w_3)A_3~.
\]
We specify these values in the following:
\begin{eqnarray*}
w_3&=&\lim_{n\rightarrow \infty}\frac{1}{n}\sum_{u=1}^n(I+A_n(\lambda_m))_{uu}(I+A_n(\lambda_{m^{'}}))_{uu}\\
&=&1+\lim_{n\rightarrow \infty}\frac{1}{n}\sum_{u=1}^n\Big(A_n(\lambda_m)_{uu}+A_n(\lambda_{m^{'}})_{uu}+A_n(\lambda_m)_{uu}A_n(\lambda_{m^{'}})_{uu}\Big)\\
&=&1+ym_1(\lambda_m)+ym_1(\lambda_{m^{'}})+\frac{y^2\big(1+m_1(\lambda_m)\big)\big(1+m_1(\lambda_{m^{'}})\big)}{\big(\lambda_m-y(1+m_1(\lambda_m))\big)\big(\lambda_{m^{'}}-y(1+m_1(\lambda_{m^{'}}))\big)}~,\\
&=& w(m,m^{'})~,\\
\tau_3&=&\theta_3=\lim_{n\rightarrow \infty}\frac{1}{n}\sum_{u,v=1}^n(I+A_n(\lambda_m))_{uv}(I+A_n(\lambda_{m^{'}}))_{vu}\\
&=&\lim_{n\rightarrow \infty}\frac{1}{n}\tr\big(I+A_n(\lambda_m))(I+A_n(\lambda_{m^{'}})\big)\\
&=&1+ym_1(\lambda_m)+ym_1(\lambda_{m^{'}})+y\int\frac{x^2}{(\lambda_{m^{'}}-x)(\lambda_m-x)}F_y(dx)\\
&=&1+ym_1(\lambda_m)+ym_1(\lambda_{m^{'}})+y\Bigg(\frac{\lambda_{m^{'}}}{\lambda_m-\lambda_{m^{'}}}m_1(\lambda_{m^{'}})+\frac{\lambda_m}{\lambda_{m^{'}}-\lambda_m}m_1(\lambda_m)\Bigg)~,\\
&=& \theta(m,m^{,})~,
\end{eqnarray*}
where we have used Lemma 6.1. in \cite{BaiYao08};
and
\begin{eqnarray*}
A_1&=&\E(\overline{x}_{l1}y_{l1}\overline{x}_{l^{'}1}y_{l^{'}1})-\rho(l)\rho(l^{'})=\E[\xi(i)\xi(j)\xi(i^{'})\xi(j^{'})]-\Sigma_{ij}\Sigma_{i^{'}j^{'}}~,\\
A_2&=&\E(\overline{x}_{l1}\overline{x}_{l^{'}1})\E(y_{l1}y_{l^{'}1})=\E[\xi(i)\xi(i^{'})]\E[\xi(j)\xi(j^{'})]~,\\
A_3&=&\E(\overline{x}_{l1}y_{l^{'}1})\E(\overline{x}_{l^{'}1}y_{l1})=\E[\xi(i)\xi(j^{'})]\E[\xi(j)\xi(i^{'})]~.
\end{eqnarray*}
Combine all these, we have
\begin{align*}
&~~~~\cov(R(\lambda_m)(i, j), R(\lambda_{m^{'}})(i^{'}, j^{'})\\
&=w(m,m^{'})\Big\{\E[\xi(i)\xi(j)\xi(i^{'})\xi(j^{'})]-\Sigma_{ij}\Sigma_{i^{'}j^{'}}\Big\}\\
&~~+\big(\theta(m,m^{'})-w(m,m^{'})\big)\E[\xi(i)\xi(j^{'})]\E[\xi(i^{'})\xi(j)]\\
&~~+\big(\theta(m,m^{'})-w(m,m^{'})\big)\E[\xi(i)\xi(i^{'})]\E[\xi(j)\xi(j^{'})]~.
\end{align*}
The proof of Theorem \ref{joint} is complete.
\end{proof}

\subsection{Proof of Theorem \ref{lu}}
\begin{proof}
Since $l_i$ is the extreme eigenvalue of $S_n$ and $(u_i,v_i)^T$
its corresponding eigenvector, we have
\begin{equation*}
\begin{pmatrix}
  l_iI_k-X_1X_1^{*} & -X_1X_2^{*} \\
  -X_2X_1^{*} & l_iI_{p}-X_2X_2^{*} \\
\end{pmatrix}\begin{pmatrix}
               u_i \\
               v_i \\
             \end{pmatrix}=0~,
\end{equation*}
where $u_i$ is the first $k$ components, and $v_i$ the remaining $p$ components, and this leads to
\begin{equation*}\label{vectoreq}
\left\{
  \begin{array}{l}
    (l_iI_k-X_1X_1^{*})u_{i}-X_1X_2^{*}v_i=0 \\[2mm]
    -X_2X_1^{*}u_i+(l_iI_{p}-X_2X_2^{*})v_i=0~.
  \end{array}
\right.
\end{equation*}
Consequently,
\begin{align}
&v_i=(l_iI_{p}-X_2X_2^{*})^{-1}X_2X_1^{*}u_i~,\label{u1}\\
&(l_iI_k-X_1\left(I_n+X_2^{*}(l_iI_p-X_2X_2^{*})^{-1}X_2\right)X_1^{*})u_i=0~.\label{u2}
\end{align}
\eqref{u2} is equivalent to \begin{align}\left(l_iI_k+l_i\underline{s}(l_i)\begin{pmatrix}
                                                                     a_1 & \cdots & 0 \\
                                                                     \vdots & \ddots & \vdots \\
                                                                    0 & \cdots & a_k \\
                                                                   \end{pmatrix}+o(1)\right)u_i=0~,\label{ev}
\end{align}
where \[\underline{s}(l_i)=\int\frac{1}{l_i-x}d\underline{F}_y(x),\] $\underline{F}_y(x)$ is the LSD of $X_2^{*}X_2$. Since $\underline{s}(l_i)=-1/a_i$, we have \eqref{ev} equivalent to \[\begin{pmatrix}
                                                                                    1-\frac{a_1}{a_i} & \cdots & 0 \\
                                                                                    \vdots & \ddots & \vdots \\
                                                                                    0 & \cdots & 1-\frac{a_k}{a_i} \\
                                                                                  \end{pmatrix}\begin{pmatrix}
                                                                                                 u_i(1) \\
                                                                                                 \vdots \\
                                                                                                 u_i(k) \\
                                                                                               \end{pmatrix}=0~,
\]
and that leads to
\begin{align}
u_i(1)=\cdots =u_i(i-1)=u_i(i+1)=\cdots=u_i(k)=0~.\label{ui0}
\end{align}
Moreover, combining \eqref{u1} with the fact that
\[
\begin{pmatrix}
  u_i^{*}& v_i^{*} \\
\end{pmatrix}\begin{pmatrix}
               u_i \\
               v_i \\
             \end{pmatrix}=1
\]
leads to
\begin{equation*}\label{ralation}
u_i^{*}(I_k+X_1X_2^{*}(l_iI_{p}-X_2X_2^{*})^{-2}X_2X_1^{*})u_i=1~,
\end{equation*}
which is also equivalent to
\begin{align}
u^2_i(i)[I_k+X_1X_2^{*}(l_iI_{p}-X_2X_2^{*})^{-2}X_2X_1^{*}](i,i)=1~
\end{align}
if take \eqref{ui0} into consideration.
Therefore, we have
\begin{eqnarray*}
u^2_i(i)&=&\frac{1}{1+[X_1X_2^{*}(l_iI_{p}-X_2X_2^{*})^{-2}X_2X_1^{*}](i,i)}\\
&=&\frac{1}{1+\E [X_1X_2^{*}(\lambda_i I_{p}-X_2X_2^{*})^{-2}X_2X_1^{*}](i,i)}+C\\
&=&\frac{1}{1+a_i y m_3(\lambda_i)}+C+o(\frac{1}{\sqrt n})~,
\end{eqnarray*}
where
\begin{eqnarray*}
C&=&\frac{1}{1+[X_1X_2^{*}(l_iI_{p}-X_2X_2^{*})^{-2}X_2X_1^{*}](i,i)}-\frac{1}{1+\E [X_1X_2^{*}(\lambda_i I_{p}-X_2X_2^{*})^{-2}X_2X_1^{*}](i,i)}\\
&=&\frac{1}{1+[X_1X_2^{*}(l_iI_{p}-X_2X_2^{*})^{-2}X_2X_1^{*}](i,i)}-\frac{1}{1+[X_1X_2^{*}(\lambda_i I_{p}-X_2X_2^{*})^{-2}X_2X_1^{*}](i,i)}\\
&&+\frac{1}{1+[X_1X_2^{*}(\lambda_i I_{p}-X_2X_2^{*})^{-2}X_2X_1^{*}](i,i)}-\frac{1}{1+\E [X_1X_2^{*}(\lambda_i I_{p}-X_2X_2^{*})^{-2}X_2X_1^{*}](i,i)}\\
&:=& C_1+C_2~.
\end{eqnarray*}
Next, we simplify the values of $C_1$ and $C_2$.

\begin{eqnarray*}
C_1&=&\frac{1}{1+[X_1X_2^{*}(l_iI_{p}-X_2X_2^{*})^{-2}X_2X_1^{*}](i,i)}-\frac{1}{1+[X_1X_2^{*}(\lambda_i I_{p}-X_2X_2^{*})^{-2}X_2X_1^{*}](i,i)}\\
&=&\frac{X_1X_2^{*}\left[(\lambda_i I_{p}-X_2X_2^{*})^{-2}-(l_iI_{p}-X_2X_2^{*})^{-2}\right]X_2X_1^{*}(i,i)}{\left[1+X_1X_2^{*}(l_i I_{p}-X_2X_2^{*})^{-2}X_2X_1^{*}(i,i)\right]\cdot\left[1+X_1X_2^{*}(\lambda_i I_{p}-X_2X_2^{*})^{-2}X_2X_1^{*}(i,i)\right]}~.
\end{eqnarray*}
First, consider the part in the above numerator:
\begin{align}\label{33}
&~~~~(\lambda_i I_{p}-X_2X_2^{*})^{-2}-(l_iI_{p}-X_2X_2^{*})^{-2}\nonumber\\
&=\left[(\lambda_i I_{p}-X_2X_2^{*})^{-1}-(l_iI_{p}-X_2X_2^{*})^{-1}\right]\cdot \left[(\lambda_i I_{p}-X_2X_2^{*})^{-1}+(l_iI_{p}-X_2X_2^{*})^{-1}\right]\nonumber\\
&=(l_i-\lambda_i)(\lambda_i I_{p}-X_2X_2^{*})^{-1}(l_i I_{p}-X_2X_2^{*})^{-1}\cdot \left[(\lambda_i I_{p}-X_2X_2^{*})^{-1}+(l_iI_{p}-X_2X_2^{*})^{-1}\right]~.
\end{align}
Since $\sqrt n(l_i-\lambda_i)$ has a central limit theorem with the following expression using our notation (see \cite{BaiYao08}):
\begin{align}\label{eigenclt}
&~~~~(l_i-\lambda_i)(1+a_iym_3(\lambda_i)+o(1))\nonumber\\
&=X_1(I+X_2^{*}(\lambda_i I_{p}-X_2X_2^{*})^{-1}X_2)X_1^{*}(i,i)-\E X_1(I+X_2^{*}(\lambda_i I_{p}-X_2X_2^{*})^{-1}X_2)X_1^{*}(i,i)~,
\end{align}
which implies that \eqref{33} tends to
\begin{eqnarray*}
2(l_i-\lambda_i)(\lambda_i I_{p}-X_2X_2^{*})^{-3}+o(1/\sqrt n)~.
\end{eqnarray*}
So
\begin{eqnarray}\label{c1}
C_1&=&2(l_i-\lambda_i)\frac{X_1X_2^{*}(\lambda_i I_{p}-X_2X_2^{*})^{-3}X_2X_1^{*}(i,i)}{[1+X_1X_2^{*}(\lambda_i I_{p}-X_2X_2^{*})^{-2}X_2X_1^{*}(i,i)]^2}+o(1/\sqrt n)\nonumber\\
&=&\frac{2a_iym_5(\lambda_i)}{(1+a_i y m_3(\lambda_i))^2}\cdot(l_i-\lambda_i)+o(1/\sqrt n)~.
\end{eqnarray}
And
\begin{align}\label{c2}
C_2&=\frac{1}{1+X_1X_2^{*}(\lambda_i I_{p}-X_2X_2^{*})^{-2}X_2X_1^{*}(i,i)}-\frac{1}{1+\E X_1X_2^{*}(\lambda_i I_{p}-X_2X_2^{*})^{-2}X_2X_1^{*}(i,i)}\nonumber\\[1mm]
&=-\frac{X_1X_2^{*}(\lambda_i I_{p}-X_2X_2^{*})^{-2}X_2X_1^{*}(i,i)-\E[X_1X_2^{*}(\lambda_i I_{p}-X_2X_2^{*})^{-2}X_2X_1^{*}](i,i)}{(1+X_1X_2^{*}(\lambda_i I_{p}-X_2X_2^{*})^{-2}X_2X_1^{*}(i,i))(1+\E X_1X_2^{*}(\lambda_i I_{p}-X_2X_2^{*})^{-2}X_2X_1^{*}(i,i))}\nonumber\\[1mm]
&=-\frac{X_1X_2^{*}(\lambda_i I_{p}-X_2X_2^{*})^{-2}X_2X_1^{*}(i,i)-\E[X_1X_2^{*}(\lambda_i I_{p}-X_2X_2^{*})^{-2}X_2X_1^{*}](i,i)}{(1+a_i y m_3(\lambda_i))^2}\nonumber\\
&~~~+o(1/\sqrt n)~.
\end{align}

Let
\begin{eqnarray}
&&A(\lambda):=I_n+X_2^{*}(\lambda_i I_{p}-X_2X_2^{*})^{-1}X_2~\label{la},\\
&&B(\lambda):=X_2^{*}(\lambda_i I_{p}-X_2X_2^{*})^{-2}X_2~\label{lb},
\end{eqnarray}
combining with \eqref{eigenclt}, \eqref{c1} and \eqref{c2} leads to
\begin{align}
C&=\frac{2a_iym_5(\lambda_i)}{(1+a_i y m_3(\lambda_i))^3}\cdot X_1[A-\E A]X_1^{*}(i,i)-\frac{X_1[B-\E B]X_1^{*}(i,i)}{(1+a_i y m_3(\lambda_i))^2}+o(1/\sqrt n)\nonumber\\
&=\frac{2a_iym_5(\lambda_i)}{(1+a_i y m_3(\lambda_i))^3}\cdot \frac{\xi_{1:n}[A-\E A]\xi_{1:n}^{*}(i,i)}{n}-\frac{\xi_{1:n}[B-\E B]\xi_{1:n}^{*}(i,i)}{n(1+a_i y m_3(\lambda_i))^2}+o(1/\sqrt n)~.
\end{align}
Therefore,
\begin{align*}
&~~~~\sqrt n\cdot\left(u^2_i(i)-\frac{1}{1+a_i y m_3(\lambda_i)}\right)\\
&=\frac{2a_iym_5(\lambda_i)}{(1+a_i y m_3(\lambda_i))^3}\cdot \frac{ \xi_{1:n}[A-\E A]\xi_{1:n}^{*}(i,i)}{\sqrt n}-\frac{ \xi_{1:n}[B-\E B]\xi_{1:n}^{*}(i,i)}{\sqrt n(1+a_i y m_3(\lambda_i))^2}+o(1)~,
\end{align*}
which leads to the fact that
\begin{align*}
\begin{pmatrix}
  \sqrt n\left(u^2_i(i)-\frac{1}{1+a_i y m_3(\lambda_i)}\right) \\[2mm]
  \sqrt n (l_i-\lambda_i) \\
\end{pmatrix}&=\begin{pmatrix}
                \frac{2a_iym_5(\lambda_i)}{(1+a_i y m_3(\lambda_i))^3} & \frac{-1}{(1+a_i y m_3(\lambda_i))^2} \\[2mm]
                \frac{1}{1+a_i y m_3(\lambda_i)}& 0 \\
              \end{pmatrix}\cdot \begin{pmatrix}
  \frac {1}{\sqrt n} \xi_{1:n}[A-\E A]\xi_{1:n}^{*}(i,i) \\[2mm]
  \frac {1}{\sqrt n} \xi_{1:n}[B-\E B]\xi_{1:n}^{*}(i,i) \\
\end{pmatrix}\\
&~~+o(1)~.
\end{align*}
If we denote
\[
D:=\begin{pmatrix}
                \frac{2a_iym_5(\lambda_i)}{(1+a_i y m_3(\lambda_i))^3} & \frac{-1}{(1+a_i y m_3(\lambda_i))^2} \\[2mm]
                \frac{1}{1+a_i y m_3(\lambda_i)}& 0 \\
              \end{pmatrix}~,
\]
and combining with Lemma \ref{abjoint}, we have got that
\[
\begin{pmatrix}
  \sqrt n\left(u^2_i(i)-\frac{1}{1+a_i y m_3(\lambda_i)}\right) \\[2mm]
  \sqrt n (l_i-\lambda_i) \\
\end{pmatrix}
\]
is asymptotically Gaussian with mean $\mathbf{0}$ and covariance matrix
\begin{align*}
DBD^T=\begin{pmatrix}
       v_{11}& v_{12}\\
       v_{12} & v_{22}\\
    \end{pmatrix}~,
\end{align*}
where
\begin{align*}
&v_{11}=\frac{(2a_iym_5)^2}{(1+a_iym_3)^6}B_{11}-\frac{4a_iym_5}{(1+a_iym_3)^5}B_{12}+\frac{1}{(1+a_iym_3)^4}B_{22}\\
&~~~~=\frac{a_i^2y^2(a_i^2+y-1)^2}{(a_i-1)^4(a_i-1+y)^4}\nu_4(i)+\frac{2a_i^2y((a_i+y-1)^2+ya_i^2)}{((a_i-1)^2-y)(a_i-1+y)^4}\\
&v_{12}=\frac{2a_iym_5}{(1+a_iym_3)^4}B_{11}-\frac{1}{(1+a_iym_3)^3}B_{12}\\
&~~~~=\frac{ya_i^2(a_i^2-1+y)((a_i-1)^2-y)}{(a_i-1)^4(a_i-1+y)^2}\nu_4(i)+\frac{2a_i^3y}{(a_i-1)(a_i-1+y)^2}\\
&v_{22}=\frac{1}{(1+a_iym_3)^2}B_{11}\\
&~~~~=\frac{a_i^2((a_i-1)^2-y)^2}{(a_i-1)^4}\nu_4(i)+\frac{2a_i^2((a_i-1)^2-y)}{(a_i-1)^2}
\end{align*}

\end{proof}

\section{Appendix}
\begin{lemma}\label{value}
For $a \notin [1-\sqrt y, 1+\sqrt y]$ and $\phi(a)=a+ya/(a-1) \notin [a_y,b_y]$, we have the following relationship:
\begin{eqnarray*}
&&m_0\circ \phi(a)=\frac{1}{a-1+y}~,\\
&&m_1\circ \phi(a)=\frac{1}{a-1}~,\label{m1}\\
&&m_2\circ \phi(a)=\frac{(a-1)+y(a+1)}{(a-1)[(a-1)^2-y]}~,\\
&&m_3\circ \phi(a)=\frac{1}{(a-1)^2-y}~\label{m33},\\
&&m_4\circ \phi(a)=\frac{(a-1)^2}{((a-1)^2-y)(a-1+y)^2}~,\\
&&m_5\circ \phi(a)=\frac{(a-1)^3}{((a-1)^2-y)^3}~,\\
&&m_6\circ \phi(a)=\frac{(a-1)^4[(a-1+y)^2+a^2y]}{((a-1)^2-y)^5}~,\\
&&m_3\circ \phi(a)+m_7\circ \phi(a)=\frac{a(a-1+y)(a-1)^2}{((a-1)^2-y)^3}~.
\end{eqnarray*}
\end{lemma}

\begin{proof}
(Sketch of the proof) Recall the definitions of these functions in \eqref{mi}, which can all be related to the combinations of the Stieltjes transform:
\[
m(\lambda)=\int \frac{1}{x-\lambda}dF(x)
\]
and it's derivatives. Besides,  $\underline{m}(\lambda)$ (definition and properties can be found in \cite{Bai04}) satisfies:
\[
\lambda=-\frac{1}{\underline{m}(\lambda)}+\frac{y}{1+\underline{m}(\lambda)}~,
\]
by taking derivatives on both sides with respect to $\lambda$ and combing with the relationship between  $\underline{m}(\lambda)$ and $m(\lambda)$:
\[
\underline{m}(\lambda)=ym(\lambda)-\frac{1}{\lambda}(1-y)~
\]
will lead to the result. Details of the calculations are omitted.

\end{proof}

\begin{lemma}\label{abjoint}
With the matrices $A$ and $B$ defined in \eqref{la} and \eqref{lb}, we have
\begin{equation*}
\begin{pmatrix}
  \frac {1}{\sqrt n} \xi_{1:n}[A-\E A]\xi_{1:n}^{*}(i,i) \\[2mm]
  \frac {1}{\sqrt n} \xi_{1:n}[B-\E B]\xi_{1:n}^{*}(i,i) \\
\end{pmatrix}\Longrightarrow \N \left(\begin{pmatrix}
                                        0 \\[2mm]
                                        0 \\
                                      \end{pmatrix},\begin{pmatrix}
                                                      B_{11} & B_{12} \\[2mm]
                                                      B_{12} & B_{22} \\
                                                    \end{pmatrix}
\right)~,
\end{equation*}
where
\begin{align*}
B_{11}=a_i^2w_1\nu_4(i)+2\tau_1 a_i^2~,~B_{22}=a_i^2w_2\nu_4(i)+2\tau_2 a_i^2~,~B_{12}=a_i^2w_3\nu_4(i)+2\tau_3 a_i^2~,
\end{align*}
and
\begin{align*}
&w_1=\frac{(a_i-1+y)^2}{(a_i-1)^2}~,~w_2=\frac{y^2}{((a_i-1)^2-y)^2}~,~w_3=\frac{y(y+a_i-1)}{(a_i-1)\cdot\left((a_i-1)^2-y\right)}\\
&\tau_1=\frac{(a_i-1+y)^2}{(a_i-1)^2-y}~,~\tau_2=\frac{y(a_i-1)^4((a_i-1+y)^2+a_i^2y)}{((a_i-1)^2-y)^5}~,~\tau_3=\frac{a_iy(a_i-1+y)(a_i-1)^2}{((a_i-1)^2-y)^3}
\end{align*}
\end{lemma}

\begin{proof}
Using Corollary \ref{real}, and let $X(1)^{*}=Y(1)^{*}=(\xi_{i1}, \cdots, \xi_{in})$ ($i$-th row of $\xi_{1:n}$), $l=l^{'}=1$ and $K=1$, we have
\begin{align*}
&A_1=\E \xi_i^4-(\E \xi_i^2)^2=a_i^2(3+\nu_4(i))-a_i^2~,\\
&A_2=\E \xi_i^2 \E \xi_i^2=a_i^2~,\\
&A_3=\E \xi_i^2 \E \xi_i^2=a_i^2~.
\end{align*}
We only have to calculate these values of $w_i$ and $\tau_i$.

First,
\begin{align*}
w_1&=\lim_{n \rightarrow \infty}\frac 1n \sum_{i=1}^n \left(1+X_2^{*}(\lambda I_{p}-X_2X_2^{*})^{-1}X_2(i,i)\right)^2\\
&=1+\left(\frac{y(1+m_1(\lambda))}{\lambda-y(1+m_1(\lambda))}\right)^2+2ym_1(\lambda)\\
&=\left(\frac{a_i+y-1}{a_i-1}\right)^2~,
\end{align*}
and
\begin{align*}
\theta_1&=\tau_1=\lim_{n \rightarrow \infty} \frac 1n \tr \left(I_n+X_2^{*}(\lambda I_{p}-X_2X_2^{*})^{-1}X_2\right)^2\\
&=1+2ym_1(\lambda)+ym_2(\lambda)\\
&=\frac{(a_i-1+y)^2}{(a_i-1)^2-y}~
\end{align*}
has been proven in \cite{BaiYao08}.

Next,
\begin{align*}
w_2=\lim_{n \rightarrow \infty}\frac 1n \sum_{i=1}^n [B(\lambda)(i,i)]^2=\lim_{n \rightarrow \infty}\frac 1n \sum_{i=1}^n \left[X_2^{*}(\lambda I_{p}-X_2X_2^{*})^{-2}X_2(i,i)\right]^2~.
\end{align*}
Since
\begin{align}\label{32}
X_2^{*}(\lambda I_{p}-X_2X_2^{*})^{-2}X_2(i,i)=e_i^{*}X_2^{*}(\lambda I_{p}-X_2X_2^{*})^{-2}X_2e_i~,
\end{align}
where $e_i$ is the column vector with its $i$-th coordinate being 1.
Recall that
\[
X_2=\frac1{\sqrt{n}}(\eta_1, \cdots, \eta_n)_{p \times n}:=\frac1{\sqrt{n}} \eta_{1:n}~.
\]
then \eqref{32} reduces to
\begin{align}\label{34}
\frac 1n \eta_i^{*}(\lambda I_{p}-X_2X_2^{*})^{-2}\eta_i~.
\end{align}
Denote $X_{2i}$ as the matrix that removing the $i$-th column of $X_2$:
\[
X_{2i}=\frac {1}{\sqrt n}(\eta_1, \cdots, \eta_{i-1}, \eta_{i+1}, \cdots, \eta_n)~,
\]
then
\[
X_2X_2^{*}=X_{2i}X_{2i}^{*}+\frac 1n \eta_i \eta_i^{*}~.
\]
Using the matrix identity that
\[
(\lambda I_{p}-X_2X_2^{*})^{-1}-(\lambda I_{p}-X_{2i}X_{2i}^{*})^{-1}=(\lambda I_{p}-X_2X_2^{*})^{-1}\frac 1n \eta_i \eta_i^{*} (\lambda I_{p}-X_{2i}X_{2i}^{*})^{-1}~,
\]
we have
\[
(\lambda I_{p}-X_2X_2^{*})^{-1}=\frac{1}{1-\frac 1n \eta_i^{*}(\lambda I_{p}-X_{2i}X_{2i}^{*})^{-1}\eta_i}\cdot(\lambda I_{p}-X_{2i}X_{2i}^{*})^{-1}~,
\]
which leads to
\[
(\lambda I_{p}-X_2X_2^{*})^{-2}=\frac{1}{(1-\frac 1n \eta_i^{*}\left(\lambda I_{p}-X_{2i}X_{2i}^{*})^{-1}\eta_i\right)^2}\cdot(\lambda I_{p}-X_{2i}X_{2i}^{*})^{-2}~,
\]
and \eqref{34} equals to
\[
\frac{\frac 1n \eta_i^{*}(\lambda I_{p}-X_{2i}X_{2i}^{*})^{-2}\eta_i}{(1-\frac 1n \eta_i^{*}\left(\lambda I_{p}-X_{2i}X_{2i}^{*})^{-1}\eta_i\right)^2}~,
\]
which tends to the limit:
\[
\frac{y \int \frac{1}{(\lambda-x)^2}dF(x)}{(1-y\int \frac{1}{\lambda-x}dF(x))^2}=\frac{ym_4(\lambda)}{(1-ym_0(\lambda))^2}~.
\]
Therefore,
\[
w_2=\frac{(ym_4(\lambda))^2}{(1-ym_0(\lambda))^2}=\frac{y^2}{((a_i-1)^2-y)^2}~.
\]

\begin{align*}
w_3&=\lim_{n \rightarrow \infty}\frac 1n \sum_{i=1}^n A(\lambda)(i,i)B(\lambda)(i,i)\\
&=\lim_{n \rightarrow \infty}\frac 1n \sum_{i=1}^n \left(1+X_2^{*}(\lambda I_{p}-X_2X_2^{*})^{-1}X_2(i,i)\right)\cdot X_2^{*}(\lambda I_{p}-X_2X_2^{*})^{-2}X_2(i,i)\\
&=\lim_{n \rightarrow \infty}\frac 1n \sum_{i=1}^n X_2^{*}(\lambda I_{p}-X_2X_2^{*})^{-1}X_2(i,i)\cdot X_2^{*}(\lambda I_{p}-X_2X_2^{*})^{-2}X_2(i,i)\\
&~~~~+\lim_{n \rightarrow \infty}\frac 1n \tr \left[X_2^{*}(\lambda I_{p}-X_2X_2^{*})^{-2}X_2\right]\\[3mm]
&=\frac{y(1+m_1(\lambda))}{\lambda-y(1+m_1(\lambda))}\cdot \frac{ym_4(\lambda)}{(1-ym_0(\lambda))^2}+ym_3(\lambda)\\
&=\frac{y(y+a_i-1)}{(a_i-1)((a_i-1)^2-y)}
\end{align*}

\begin{align*}
\theta_2&=\tau_2=\lim_{n \rightarrow \infty} \frac 1n \sum_{i,j=1}^n \left(X_2^{*}(\lambda I_{p}-X_2X_2^{*})^{-2}X_2(i,j)\right)^2\\
&=\lim_{n \rightarrow \infty} \frac 1n \tr \left[X_2^{*}(\lambda I_{p}-X_2X_2^{*})^{-2}X_2\right]^2\\
&=y\int \frac{x^2}{(\lambda-x)^4}dF(x)\\
&=ym_6(\lambda)\\
&=\frac{y(a_i-1)^4\left((a_i-1+y)^2+a_i^2y\right)}{((a_i-1)^2-y)^5}
\end{align*}

\begin{align*}
\theta_3&=\tau_3=\lim_{n \rightarrow \infty} \frac 1n \tr [A(\lambda)B(\lambda)]\\
&=\lim_{n \rightarrow \infty} \frac 1n \tr \left\{\left(I_n+X_2^{*}(\lambda I_{p}-X_2X_2^{*})^{-1}X_2\right)X_2^{*}(\lambda I_{p}-X_2X_2^{*})^{-2}X_2\right\}\\
&=y \int \frac{x}{(\lambda-x)^2}dF(x)+y\int \frac{x^2}{(\lambda-x)^3}dF(x)\\
&=y(m_3(\lambda)+m_7(\lambda))\\
&=\frac{a_iy(a_i-1+y)(a_i-1)^2}{((a_i-1)^2-y)^3}
\end{align*}

The proof of Lemma \ref{abjoint} is complete.
\end{proof}

\section*{Acknowledgement}
We thank the anonymous referees for helpful comments. In particular, the application in Proposition \ref{propan} has been suggested by one of the referees.

\end{document}